\documentclass[reqno,10pt]{amsart}

\usepackage[margin=0.9in]{geometry}

\usepackage{amssymb,amsmath,amsthm}
\usepackage{amscd}
\usepackage{amsfonts}
\usepackage{amssymb}
\usepackage{latexsym}
\usepackage{xcolor}
\usepackage{esint}

\usepackage{hyperref}

\usepackage{graphicx}
\graphicspath{ {images/} }

\usepackage[makeroom]{cancel}

\newtheorem{theorem}{Theorem}

\newtheorem{definition}{Definition}
\newtheorem{lemma}{Lemma}
\newtheorem{proposition}[theorem]{Proposition}
\newtheorem{remark}{Remark}
 \newtheorem*{theorem*}{Rough version of the Main theorem}

\let\e=\varepsilon

\let\p=\partial

\let\O=\Omega

\let\pll = \parallel

\numberwithin{equation}{section}

\let\hide\iffalse

\DeclareMathAlphabet{\mathpzc}{OT1}{pzc}{m}{it}

\newcommand{\R}{\mathbb{R}}

\newcommand{\be}{\begin{equation}}
\newcommand{\bm}{\begin{multline}}
\newcommand{\ee}{\end{equation}}
\newcommand{\dd}{\mathrm{d}}

\newcommand{\xb}{x_{\mathbf{b}}}
\newcommand{\tb}{t_{\mathbf{b}}}
\newcommand{\vb}{v_{\mathbf{b}}}

\newcommand{\xf}{x_{\mathbf{f}}}
\newcommand{\tf}{t_{\mathbf{f}}}
\newcommand{\vf}{v_{\mathbf{f}}}

\newcommand{\vbn}{v_{\mathbf{b}, 3}}
\newcommand{\vfn}{v_{\mathbf{f}, 3}}

\newcommand{\Bes}{\begin{eqnarray*}}
\newcommand{\Ees}{\end{eqnarray*}}
\newcommand{\Be}{\begin{equation} }
\newcommand{\Ee}{\end{equation}}
\newcommand{\Bs}{\begin{split}}

\newcommand{\A}{\mathcal{A}}

\newcommand{\vertiii}[1]{{\left\vert\kern-0.25ex\left\vert\kern-0.25ex\left\vert #1 
    \right\vert\kern-0.25ex\right\vert\kern-0.25ex\right\vert}}

 \makeatletter
\def\munderbar#1{\underline{\sbox\tw@{$#1$}\dp\tw@\z@\box\tw@}}
\makeatother

\def\p{\partial}

\def\O{\Omega}
\def\R{\mathbb{R}}

\def\B{\begin{equation}}
\def\E{\end{equation}}
\def\BN{\begin{eqnarray*}}
\def\EN{\end{eqnarray*}}

\def\bcb{\begin{color}{blue}}
\def\ec{\end{color}}

\def\bcr{\begin{color}{red}}
\def\ec{\end{color}}

\begin{document}

\title{Boundary effect under 2D Newtonian gravity potential in the phase space}

\author{Jiaxin Jin}
\address{Department of Mathematics, The Ohio State University, Columbus, OH, 43210, USA, email: jin.1307@osu.edu}
\author{Chanwoo Kim}
\address{Department of Mathematics, University of Wisconsin-Madison, Madison, WI, 53706, USA, email: chanwoo.kim@wisc.edu, chanwookim.math@gmail.com}

\maketitle

 \begin{abstract}
 	We study linear two-half dimensional  Vlasov equations under the logarithmic gravity potential in the half space of diffuse reflection boundary. We prove decay-in-time of the exponential moments with a polynomial rate, which depends on the base logarithm.

 \end{abstract}






\section{Introduction}

In this paper, we consider a free molecules without 
{\color{black} intermolecular}
interaction which are contained in 
{\color{black}
a horizontally-periodic three dimensional half-space $\O = \mathbb{T}^2 \times \R_{+}$, }
and subjected to the gravity field. A governing kinetic model of the system is the Vlasov equations:
\Be \label{equation for F} 
\partial_t F + v \cdot \nabla_{x} F - \nabla \Phi (x) \cdot \nabla_{v} F = 0, \ \text{for} \   (t,x,v) \in \R_{+} \times \Omega \times \R^3.
\Ee
Here, $\Phi(x)$ is a given external field (gravity), which will be specified later in \eqref{field property}.

At the bottom of domain, 
{\color{black}
the phase boundary $\gamma:=  \{ (x, v) \in \partial \Omega \times \mathbb{R}^3\}$ is decomposed into the outgoing boundary and incoming boundary $ \gamma_{\pm} := \{ (x, v) \in \partial \Omega \times \mathbb{R}^3, n(x) \cdot v \gtrless 0 \}$ with the outward normal $n(x)$ at $x \in \p\O$. It is clear that $|\p\O| = 1$.
Further, we consider
}
the molecules interact with the boundary thermodynamically via a \textit{diffusive reflection} boundary condition
\Be \label{diff_F}
F(\cdot, x, v)  
= \mu (x, v) \int_{n(x) \cdot v^1 >0} F(\cdot, x, v^1) \{ n(x) \cdot v^1 \} \dd v^1 \ \ \text{for} \ \  (x,v) \in \gamma_- : = \{ x \in \p\O \ \text{and} \    v_3 > 0\},
\Ee   
such that an outgoing distribution is proportional to the thermal equilibrium of the unit boundary temperature:
\Be \label{wall_M}
\mu (x, v) = \frac{1}{2 \pi}e^{-  \frac{|v|^2}{2}}
\ \ \ \ \ \ \ \ \ \ \ \   \text{(wall Maxwellian)}.
\Ee  
where
$\int_{n(x) \cdot v^1 >0} \mu (x, v^1) 
\{n(x) \cdot v^1\} \dd v^1 = 1$, and we have a null flux at the boundary and enjoy the conservation of total mass:
\Be
\iint_{\O \times \R^3} F(t,x,v) \dd x \dd v  = \iint_{\O \times \R^3} F(0,x,v) \dd x \dd v  = \mathfrak{m}>0.
\Ee 
Throughout this paper, we always assume that the total mass equals $\mathfrak{m}$.

If the boundary temperature varies with the position on the boundary, 
then stationary solutions to \eqref{diff_F} are neither given by explicit formulas nor are equilibria (local Maxwellian) in general, if they exist (see \cite{JK2} for the construction of steady solutions). 
{\color{black}
This is because any explicit solution can be obtained by backtracking along the characteristics until the boundary.
Under the non-isothermal case when $\mu_{\theta} (x, v) = \frac{1}{2 \pi}e^{-  \frac{|v|^2}{2 \theta (x)}}$ and $\theta (x)$ varies with $x$, local Maxwellian doesn't satisfy the diffusive boundary condition in general.
}

In this paper, we only focus on the asymptotic stability of simpler isothermal boundary for the sake of simplicity. In this case of the isothermal boundary \eqref{wall_M}, a stationary solution has an explicit form: for some $c_\mathfrak{m}>0$
\Be
\tilde{\mu} (x, v) :=  \frac{c_{\mathfrak{m}}}{2 \pi}e^{- ( \frac{|v|^2}{2} + \Phi (x) ) }.
\Ee
The uniqueness of stationary problem can be easily proved as the problem is linear (see \cite{JK2} for the details). 


The main interest in this paper is to study stabilizing effect of the diffusive reflection boundary to the Vlasov equations under the logarithmic potential
\Be \label{field property_1}
\Phi (x) = \log_a (1 + x_3).
\Ee
This potential is physically relevant in the 2D universe. Indeed the logarithmic potential \eqref{field property_1} corresponds to the Newtonian potential in the 2-dimensional universe. A relevant model is the two-half dimensional Vlasov equation: 
{\color{black}
\Be \label{2.5}
\p_t F + \sum_{i=1,3} v_i \p_{x_i} F - \p_{x_3} \Phi (x) \p_{v_3} F = 0, 
\Ee
where the spatial domain is $\mathbb{T} \times \R_{+}:=\{ (x_1, x_3) \in \mathbb T \times \mathbb R: x_3 >0 \}$. 
}

Our full 3 dimensional problem \eqref{equation for F} can directly apply to this two-half dimensional model \eqref{2.5} by setting data homogeneous in $x_2$-direction, that is, $F = F(t,x_1, x_3, v)$ and $F_0 = F_0 (x_1, x_3, v)$ in the spatial domain $\O = \mathbb{T} \times \R_{+}$ and the domain of the velocities is still $\R^3$. 

 


%
%

\bigskip

{\color{black}
\noindent \textbf{Notations.} Here we clarify some notations: 
$A \lesssim B$ if $A\leq C B$ for a constant $C >0$ which is independent on $A,B$;
$A \lesssim_\theta B$ if $A\leq C B$ for a constant $C=C(\theta)>0$ which depends on $\theta$ but is independent on $A,B$;
$\| \cdot \|_{L^1_{x,v}}$  for the norm of $L^1(\O \times \R^3)$; $\| \cdot \|_{L^\infty_{x,v}}$ or $\| \cdot \|_\infty$  for the norm of $L^\infty(\bar{\O} \times \R^3)$; $|g|_{L^1_{\gamma_\pm}}=\int_{\gamma_\pm} |g(x,v) | |n(x) \cdot v| \dd S_x \dd v$ where $\dd S_x = \dd x_1 \dd x_2$ represents the measure on the boundary $\p\O$ and $n(x)$ is the outward normal at $x \in \p\O$; an integration $\int_Y f(y)\dd y$ is often abbreviated to $\int_Y f$, if it is not ambiguous. 
Finally, we remark that $n$ represents an integer without $x \in \p\O$ (e.g. Proposition \ref{prop:mapV}).
}

\bigskip

\noindent \textbf{Main Theorems. } 
The main interest in this work is to study a long-time behavior of solutions to the Vlasov equations for the field as follows:
\Be \label{field property}
\Phi (x) = \log_a (1 + x_3),
\ \text{ and } \ 
\A = \left[\frac{1}{\ln (a)}\right] \geq 8,
\ee
where $[ m ]$ represents the biggest integer less than or equal to $m$. 
{\color{black}
Here we set $\A$ as the integer part of $1 / \ln (a)$ for the convenience of decay rates in main results (see Theorem \ref{theorem_1} and Theorem \ref{theorem}).
}

The gravitational potential in the logarithm form plays an important role to the convergence speed which turns out a polynomial rate depends on the base of the logarithm. 

We express the perturbation form as 
\Be \label{pert}
F(t,x,v)= \tilde{\mu} (x, v) + f (t,x,v),
\Ee
and the initial data $F_0(x,v)= \tilde{\mu} (x, v) + f_0 (x,v)$.

\medskip

Theorem \ref{theorem_1} shows $L^1$-estimates on every fluctuation which is of zero initial mass.


\begin{theorem} \label{theorem_1} 
Consider the initial data $F_0 (x, v) = \tilde \mu(x,v) + f_0 (x,v) \geq 0$, such that 
\be	\label{f=0_mass}
\iint_{\O \times \R^3} f_0 (x,v) \dd x \dd v = 0, \ \
\| e^{ \frac{1}{2} |v|^2+ \Phi (x) } f_0\|_{L^\infty_{x,v}} < \infty.
\ee	
There exists a unique global-in-time solution 
	\Be	\label{def:f}
	F(t,x,v) =\tilde \mu(x,v) + f(t,x,v) \geq 0
	\Ee
	to \eqref{equation for F} and the boundary condition \eqref{diff_F} with the initial condition $F(t,x,v)|_{t=0} = F_0(x,v)$ in $\O \times \R^3$, such that
	\Be \label{cons_mass_f} 
	\iint_{\Omega \times \R^3} f (t, x, v) \dd x \dd v = 0, \ \ \text{for all } t\geq 0.
	\Ee
Moreover, we have 
{\color{black}
\Be \label{est:theorem_1} 
\| f(t) \|_{L^1_{x,v}} 
\leq C (\ln\langle t\rangle )^{\A - 6 - \frac{\delta}{2}} \langle t\rangle^{-( \A-6)} \times \| e^{ \frac{1}{2} |v|^2+ \Phi (x) } f_0\|_{L^\infty_{x,v}},
\Ee
where $C = C (\O)$ only depends on the domain $\O$, $0 < \delta < 1$ and $\A$ is given as in \eqref{field property}.
}
\end{theorem}

{\color{black}
\begin{remark}
To prove Theorem \ref{theorem_1}, we introduce and compute the norms of $f (t,x,v)$ at time $t = k T_0$ with $k \in \mathbb{N}$ (see \eqref{|||i}). Further, the time interval $T_0$ depends only on the domain $\O$ (see Propositions \ref{prop:Doeblin} and \ref{prop:energy}). Therefore, 
the constant $C$ only depends on the domain $\O$.
\end{remark}
}

\smallskip

Theorem \ref{theorem} proves the the decay of the exponential moment on the fluctuation. 

\begin{theorem} \label{theorem}
Assume all conditions in Theorem \ref{theorem_1}.
For all $t \geq 0$ and $0 \leq 2 \theta < \theta^\prime = \frac{1}{2}$,
\Be \label{theorem_infty_1}
 	\sup_{t\geq0} \| e^{\theta^\prime (|v|^2+ 2\Phi (x))} f (t)\|_{L^\infty_{x,v}} \lesssim \| e^{\theta^\prime (|v|^2+ 2\Phi (x))} f_0\|_{L^\infty_{x,v}}.
\Ee
\Be \label{theorem_infty}
\sup_{x \in \bar{\O}}\int_{\R^3} e^{\theta  (|v|^2+ 2\Phi (x))} |f(t,x,v) |\dd v  
\lesssim_{\theta} \langle t\rangle^{7 - \A}.
\Ee
\end{theorem} 

\begin{remark}
The decay rate and the potential have a close relation. When the gravity is constant then the system has an exponential decay \cite{JK2, KIM_VP}. 
{\color{black}
On the other hand, when the domain is bounded and the potential is zero, the decay rate is polynomial depending on the spatial dimension. This is due to the fact that low velocities stay in the system for a long time. About this direction we refer to \cite{AG, B, JK1, Lods} and the references therein. 
}

\end{remark}

\bigskip

\noindent \textbf{Difficulties and Ideas. }
{\color{black}
Throughout this paper, we use the fundamental idea where for each velocity obtained from the diffusive reflection boundary condition, we compute how the velocity transfers through space under the kinetic operator. This idea is realized by \textit{the stochastic cycles}.
}

The characteristics of \eqref{equation for F} are determined by the Hamilton ODEs
\Be \label{characteristics}
\begin{cases}
	\frac{\dd}{\dd s} X(s; t, x, v) = V(s; t, x, v),
	\\
	\frac{\dd}{\dd s} V(s; t, x, v) = - \nabla \Phi (X(s; t, x, v)),
\end{cases}  
\Ee
for $- \infty < s, t < \infty$ with $(X(t; t, x, v), V(t; t, x, v)) = (x, v)$.

\begin{definition}[Stochastic Cycles] \label{def_cycles}

{\color{black}
Consider $(X,V)$ solving \eqref{characteristics}, which is the characteristics of the Vlasov equations \eqref{equation for F}.
}
Define the backward exit time $t_{\mathbf{b}}$ and the forward exit time $\tf$, 
	\Be \label{def_tb}
	\begin{split} 
		& t_{\mathbf{b}}(x, v) := \sup \{s \geq 0: X(t - \tau; t, x, v) \in \Omega, \  \forall\tau \in [0, s) \}, \
		x_{\mathbf{b}}(x, v) := X(t - t_{\mathbf{b}}(x, v); t, x, v), 
		\\& t_{\mathbf{f}}(x, v) := \sup \{s \geq 0: X(t + \tau; t, x, v) \in \Omega, \  \forall\tau \in [0, s) \}, \
		x_{\mathbf{f}}(x, v) := X(t + t_{\mathbf{f}}(x, v); t, x, v).
	\end{split} 
	\Ee
	We define the stochastic cycles: 
	\be \notag
	t^1 (t, x, v) = t - t_{\mathbf{b}}(x, v),
	\ x^1 (x, v) = x_{\mathbf{b}}(x, v) = X(t^1, t, x, v), \ \vb(x, v) =  V(t^1, t, x, v),
	\ee
	\Be \label{def:t_k}
	\begin{split}
		& t^k (t, x, v, v^1,..., v^{k-1}) = t^{k-1}
		- t_{\mathbf{b}}(x^{k-1}, v^{k-1}), \ 
		\tb^k = \tf^{k+1} = t^k - t^{k+1},
		\\&  x^k (t, x, v, v^1,..., v^{k-1}) = X(t^{k}; t^{k-1}, x^{k-1}, v^{k-1}), \ 
		\vb^k= V(t^{k+1}; t^k, x^k, v^k),
	\end{split}
	\Ee
	where we define $v^j \in \mathcal{V}_j := \{v^j \in \R^3: n(x^j) \cdot v^j > 0 \}$ with the measure 
	$\dd \sigma_j = \dd \sigma_j (x^j)$ on $\mathcal{V}_j$ which is given by
	\Be \label{def:sigma measure}
	\dd \sigma_j := \mu (x^{j+1}, \vb^{j}) \{ n(x^j) \cdot v^j \} \dd v^j.
	\Ee
Here, $n(x)$ is the outward normal at $x \in \p\O$.
\end{definition}

{\color{black}
Given $(t, x, v) \in \R_+ 
\times \O \times \R^3$, suppose that $(X(s; t, x, v), V(s; t, x, v))$ solves \eqref{characteristics}, the backward exit time $\tb$ stands for the longest backward time, for which the characteristic $X(s; t, x, v)$ stays in the domain $\O$. And $\xb = X(t - \tb; t, x, v)$ is the boundary position when $s = t -\tb$.
Similarly, the forward exit time $\tf$ is the longest forward time, for which the characteristic $X(s; t, x, v)$ stays in the domain $\O$, and $\xf = X(t + \tf; t, x, v)$ is the boundary position when $s = t + \tf$.
Moreover, since the field $\Phi(x)$ is timely independent, this leads that both $\tb$ and $\tf$ are also timely independent.
}

Now we explain a major difficulty in the presence of logarithmic potential. Compared to the constant potential (for example: $gx_3$) considered in \cite{JK2}, the backward exit time $\tb$ and the forward exit time $\tf$ have much weaker control. Indeed we can derive that, for any $(x,v) \in \gamma_-$,  
\Be \notag
a^{\frac{1}{2} |v_3|^2} \sqrt{1 - a^{- \frac{1}{2} |v_3|^2}} 
\lesssim \tb (x, v)
\lesssim a^{\frac{1}{2} |v_3|^2},
\Ee
using the conservation of mass on the characteristic line crucially. This control shows that the backward exit time $\tb$ is comparable to $a^{\frac{1}{2} |v_3|^2}$ when $n(x) \cdot v \gg 1$. The crucial observation is that the Maxwellian $\mu (x, v) = \frac{1}{2 \pi}e^{-  \frac{|v|^2}{2}}$ has a polynomial control on $\tb$ (or $\tf$ for $(x,v) \in \gamma_+$) depending on $\A = [\frac{1}{\ln (a)}]$.
Therefore we are able to control the sum of infinite Maxwellian terms produced by the periodic domain (see Lemma \ref{lem: sum of mu}).

The proof of dynamical stability on the fluctuations $f (t, x, v)$, which solves \eqref{equation for F}, \eqref{diff_F}, and \eqref{f=0_mass}, is based 
{\color{black}
on a lower bound with the unreachable defect (see Proposition \ref{prop:Doeblin 1}) as follows:
}
\Be \notag
f(NT_0,x,v) \geq  
\mathfrak{m}(x,v) \Big\{
\iint_{\O \times \R^3}f((N-1)T_0,x,v) \dd v \dd x 
-  \iint_{\O \times \R^3} \mathbf{1}_{t_\mathbf{f}(x,v)\geq \frac{T_0}{4}} f((N-1)T_0,x,v)  \dd v \dd x 
\Big\},
\Ee
{\color{black}
where $\mathfrak{m}(x,v)$ is defined in \eqref{def:m}.
This is also considered as the Doeblin condition where $f(t, x, v)$ is bounded below by the part of the mass of molecules in previous stochastic cycles.
We refer \cite{CANIZO2023109830}, which includes a systematic exposition of Doeblin-type arguments.
}

Next we control the unreachable defect
{\color{black}
(see Lemma \ref{lemma:energy}).
}
Since the forward exit time under Vlasov operator can be controlled as follows:
\Be \notag
\frac{\partial}{\partial t} \tf (t, x, v) + v \cdot \frac{\partial}{\partial x} \tf (t, x, v) - \nabla \Phi (x) \cdot \frac{\partial}{\partial v} \tf (t, x, v) = -1,
\Ee
any weight function $\varphi(\tf)$ satisfies
$( v \cdot \nabla_x
- \nabla \Phi (x) \cdot \nabla_v) \varphi (\tf) = -\varphi^\prime(\tf)$.
Hence, we derive that 
{\color{black}
given a real function $\varphi (\tau )$, it satisfies that for any $\tau \geq 0$, $\varphi (\tau ) \geq0$, $\varphi^\prime \geq0$, and
}
\Be \label{cond:varphi_intro} 
\int_1^\infty \tau^{3 - \A} \varphi(\tau) \dd \tau < \infty. 
\Ee
{\color{black}
It is worth to compare to the constant gravity case \cite{JK2} when we allow $\int_1^\infty e^{- \frac{1}{2} \tau^2} \varphi(\tau) \dd \tau < \infty$ and then the system has an exponential decay. This weaker weight in $\tau$ restricts the range of $\varphi$ and consequently deduces a polynomial decay.
}

Suppose $f$ solves \eqref{equation for F} and \eqref{diff_F}, there exists $C>0$ independent of $t_*$, $t$, such that for all $0 \leq t_* \leq t$, 
\Be \notag
\begin{split}
& \ \ \ \ \| \varphi(\tf) f(t) \|_{L^1_{x,v}}
  + \int^{t}_{t_*}
  \| \varphi^\prime(\tf) f  \|_{L^1_{x,v}} \dd s
  +  \int_{t_*}^{t} | \varphi(\tf) f|_{L^1_{\gamma_+}} \dd s
\\& \leq \| \varphi(\tf) f(t_*) \|_{L^1_{x,v}} +
 C (t - t_* + 1) \| f(t_*) \|_{L^1_{x,v}} + \frac{1}{4}  \int^{t}_{t_*} |f |_{L^1_{\gamma_+}} \dd s.
\end{split}
\Ee 
We remark that the exponent $3-\A$ in \eqref{cond:varphi_intro} is determined from the initial condition $\| e^{ \frac{1}{2} |v|^2+ \Phi (x) } f_0\|_{L^\infty_{x,v}} < \infty$ and polynomial control between $\mu (x, v)$ and $\tf$ for $(x,v) \in \gamma_+$. Furthermore, this exponent will restrict the decay rate of Theorem \ref{theorem}.
Then we introduce two norms $\vertiii{\cdot}_2$ and $\vertiii{\cdot}_4$ as
\Be \notag
\vertiii{f}_i :=
\|f \|_{L^1_{x,v}}
+ \frac{ 4 \mathfrak{m}_{T_0} }{ \varphi_{i-1} (\frac{3T_0}{4})} \| \varphi_{i-1}(\tf) f\|_{L^1_{x,v}}  
+ \frac{ 4e \mathfrak{m}_{T_0} }{ T_0 \varphi_{i-1} (\frac{3T_0}{4})} \| \varphi_{i}(\tf) f\|_{L^1_{x,v}},
\Ee
where four polynomial weights $\varphi_1, \varphi_2, \varphi_3, \varphi_4$ are defined in \eqref{varphis}.
We derive the polynomial decay in $L^1$ after using an energy estimate on these norms. 


At last, to conclude a pointwise bound on the exponential moment, we introduce several weight functions $\varrho (t)$ and $w^\prime (x, v)$. 
Then we control the bound on $\varrho (t) w^\prime (x, v) f(t, x, v)$ via stochastic cycles expansions and polynomial decay on the fluctuations proved before.
This allows us to conclude the decay of the exponential moment.

\bigskip

\noindent \textbf{Structural of the paper. }
For the rest of paper, we collect some basic preliminaries in Section \ref{sec:background}.
Then in Section \ref{sec: L1 estimate}, we study the weighted $L^1$-estimates and prove Theorem \ref{theorem_1}. Finally in Section \ref{sec: exponential moments}, we show an $L^\infty$-estimate of moments in Theorem \ref{theorem}.


\section{Background}
\label{sec:background}



We first list some properties for \eqref{characteristics}, the characteristics of \eqref{equation for F}. 

\begin{lemma}[\cite{JK2}] \label{lem:COV}
For any $g (t, x, v)$ and $(X,V)$ solving \eqref{characteristics}, we have   
	\begin{align}
		\int_{\gamma_{+}} \int_0^{t_{-}} g(t,X(t,t+s,x,v),V(t,t+s,x,v)) |n(x) \cdot v|\dd s \dd v \dd S_x 
		= \iint_{\O \times \R^3} g(t,y,v) \dd y \dd v, 
		\label{COV} \\
		\int_{\gamma_{-}} \int_0^{t_{+}} g(t,X(t,t-s,x,v),V(t,t-s,x,v)) |n(x) \cdot v|\dd s \dd v \dd S_x 
		= \iint_{\O \times \R^3} g(t,y,v) \dd y \dd v, 
		\label{COV+} \\
		\int_{\gamma_{\pm}}
		g(t,x_{\mp}(x, v),v_{\mp}(x, v))
		|n(x) \cdot v| \dd v \dd S_x
		= \int_{\gamma_{\mp}}
		g(t,y,v)
		|n(y) \cdot v|\dd v \dd S_y. \label{COV_bdry}
	\end{align}
	Here, for the sake of simplicity, we have abused the notations temporarily: $t_-=\tb, x_-= \xb$ and $t_+= \tf, x_+= \xf$.  
\end{lemma} 

The following Lemma will let us derive the stochastic cycles.

\begin{lemma}[\cite{JK2}] \label{sto_cycle_1}
Suppose $F (x, v)$ solves \eqref{equation for F} and \eqref{diff_F} with $0 \leq t_* \leq t$, then for $k \geq 1$,
	\begin{align}
		F (x, v) = 
		& \mathbf{1}_{t^1 < t_*}
		F (X(t_*; t, x, v), V(t_*; t, x, v)) \label{expand_F1} 
		\\&  + \mu (x^1, \vb) \sum\limits^{k-1}_{i=1}  \int_{\prod_{j=1}^{i} \mathcal{V}_j}   
		\Big\{   \mathbf{1}_{t^{i+1} < t_* \leq t^{i }} F (X(t_*; t^i, x^i, v^i), V(t_*; t^i, x^i, v^i))  \Big\}
		\dd  \Sigma_{i}
		\label{expand_F2}
		\\& + \mu (x^1, \vb) \int_{\prod_{j=1}^{k } \mathcal{V}_j}   
		\mathbf{1}_{t^{k } \geq t_* }
		F (x^{k }, v^{k })
		\dd  \Sigma_{k}
		, \label{expand_F3}
	\end{align} 
	where 
	$\dd  {\Sigma}_{i} := \frac{ \dd \sigma_{i}}{ \mu (x^{i+1}, \vb^{i})} \dd \sigma_{i-1} \cdots  \dd \sigma_1$, with $\dd \sigma_j = \mu (x^{j+1}, \vb^{j}) \{ n(x^j) \cdot v^j \} \dd v^j$ in \eqref{def:sigma measure}, and $\vb^{j} = \vb (x^j, v^j)$ defined in \eqref{def:t_k}. Here, $(X,V)$ solves \eqref{characteristics}.
\end{lemma}

{\color{black}
\begin{proof}
The proof follows from a similar argument, Lemma 2 in \cite{JK2}. 
\end{proof}
}

\begin{remark} \label{probability}
From the Lemma \ref{conservative field} and $\mu (x, v)= \mu (x, |v|)$, we obtain 
 	$\mu (x^{j+1}, \vb^{j}) = \mu (x^{j+1}, v^{j})$ with $\vb^{j} = \vb (x^j, v^j)$. 
 	Therefore, in the rest of paper, we write $\dd \sigma_j$ as
 	\Be \notag
 	\dd \sigma_j = \mu (x^{j+1}, v^{j}) \{ n(x^j) \cdot v^j \} \dd v^j,
 	\Ee
where $v^j \in  \mathcal{V}_j  := \{v^j \in \R^3: n(x^j) \cdot v^j > 0 \}$ and $(X,V)$ solves \eqref{characteristics}.
\end{remark}

\begin{lemma} \label{conservative field}
Consider $(X,V)$ solving \eqref{characteristics}, then for $v \in  \mathcal{V}  := \{v \in \R^3: n(x) \cdot v > 0 \}$, $x \in \p\O$ and $\vb = \vb (x, v)$,
\Be 
|\vb| = |v|.
\Ee
\end{lemma}

\begin{proof}

{\color{black}
The proof follows from a similar argument, Lemma 3 in \cite{JK2}.
}
Since $(X,V)$ solves \eqref{characteristics} with $v \in  \mathcal{V}$, we compute the following derivative:
\Be \label{eq:energy conservation}
\begin{split} 
& \ \ \ \ \frac{\dd}{\dd s}
\big( 
\frac{|V(s;t,x,v)|^2}{2} + \Phi(X(s;t,x,v)) 
\big)
\\& = V(s;t,x,v) \cdot \frac{\dd V}{\dd s} + \nabla \Phi \cdot \frac{\dd X}{\dd s} 
= - V(s;t,x,v) \cdot \nabla \Phi (X(s; t, x, v)) + \nabla \Phi \cdot V(s; t, x, v) = 0.
\end{split}
\Ee
Recall that
$(X(t; t, x, v), V(t; t, x, v)) = (x, v)$, $(X(t - t_\mathbf{b}; t, x, v), V(t - t_\mathbf{b}; t, x, v)) = (x_\mathbf{b}, \vb)$. By taking $s = t - \tb$ and $s = t$, we obtain
\Be \notag
|v|^2/2 + \Phi(x) = |\vb|^2/2 + \Phi(x_\mathbf{b}).
\Ee  
Since $\Phi (x)|_{x_3 = 0} = \log_a (1 + x_3)|_{x_3 = 0} \equiv 0$ and $x, \xb \in \p\O$, we have
\be \notag
\Phi(x) = \Phi(x_\mathbf{b}) = 0,
\ee
which implies $|\vb| = |v|$.
\end{proof}

{\color{black}
\begin{remark}
Since $\mu (x, v) = \frac{1}{2 \pi}e^{-  \frac{|v|^2}{2}}$ is radial, that is $\mu (x, v_1) = \mu (x, v_2)$ if $|v_1| = |v_2|$.
From the Lemma \ref{conservative field}, we obtain 
$\mu (x^{j+1}, \vb^{j}) = \mu (x^{j+1}, v^{j})$
where $\vb^{j} = \vb (x^j, v^j)$. 
Therefore, in the rest of paper, we write $\dd \sigma_j$ as
 	\Be \notag
 	\dd \sigma_j = \mu (x^{j+1}, v^{j}) \{ n(x^j) \cdot v^j \} \dd v^j,
 	\Ee
where $v^j \in  \mathcal{V}_j  := \{v^j \in \R^3: n(x^j) \cdot v^j > 0 \}$ and $(X,V)$ solves \eqref{characteristics}.\end{remark}
}

Now we consider the change of variables 
$v \in \{v \in \R^3
: n(x) \cdot v  >  0\}
\mapsto
(\xb(x,v), \tb(x,v)) \in \partial \Omega \times \R_{+}.$
Since the domain is periodic, this is a local bijective mapping. For fixed $x, \tb$ and $\xb$, we introduce the set of velocities $\{v^{m, n}\}$ with $m, n \in \mathbb{Z}$ such that
{\color{black}
\Be \label{def:vmn}
v^{m, n}
\in \{v^{m, n} \in \R^3
: n(x) \cdot v^{m, n} > 0\}
 \mapsto
(\xb, \tb):=
(\xb + (m, n, 0), \tb)
= (\xb, \tb)
\in \partial \Omega \times \R_{+}.
\ee
}

\begin{proposition} \label{prop:mapV}
Consider $(X,V)$ solving \eqref{characteristics}, 
\begin{itemize}
\item For fixed $x \in \p\O$, and $m, n \in \mathbb{Z}$, we introduce the following map:
\Be \label{mapV}
v \in \{v \in \R^3
: n(x) \cdot v  >  0\}
\mapsto (\xb, \tb):=
(\xb(x,v)  + (m, n, 0), \tb(x,v)) \in \partial \Omega \times \R_{+}.
\Ee 
Then the map \eqref{mapV} is locally bijective and has the change of variable formula as
\Be \label{jacob:mapV}
\tb^{-2} (1 + |v_3| \tb)^{-1} \dd \tb \dd S_{\xb }
\lesssim \dd v \lesssim \tb^{-2} \dd \tb \dd S_{\xb }.
\Ee


\item Similarly we have a locally bijective map: 
\Be \notag
v \in \{v \in \R^3
: n(x) \cdot v<0\}
\mapsto (\xf , \tf ):=
(\xf(x,v) + (m, n, 0), \tf(x,v)) \in \partial \Omega \times \R_{+},
\Ee 
with 
\Be \label{mapV_f}
\tf^{-2} (1 + |v_3| \tf)^{-1} \dd \tf \dd S_{\xf}
\lesssim \dd v \lesssim \tf^{-2} \dd \tf \dd S_{\xf}.
\Ee 
\end{itemize}
\end{proposition} 

\begin{proof}
We just need to show \eqref{jacob:mapV}, since \eqref{mapV_f} can be deduced after changing the backward variables into forward variables. For the sake of simplicity, we have abused the notations temporarily: 
\be \notag
\begin{split}
& \xb := x^1 = (x^1_1, x^1_2, x^1_3)= (x^1_\parallel, x^1_3), \ \
v = (v_1, v_2, v_3) = (v_\parallel, v_3), 
\\& \vb = (v_{\mathbf{b}, 1}, v_{\mathbf{b}, 2}, v_{\mathbf{b}, 3}), \ \
t_{\mathbf{b}} = \tb(x,v).
\end{split}
\ee
Recall $\Phi (x) = \log_a (1 + x_3)$, then we get $\nabla \Phi = (0, 0, \frac{1}{(1 + x_3) \ln (a)})$ with $\frac{1}{\ln (a)} > 1$.

Now we compute the determinant of the Jacobian matrix. 
Fixing $x, t$ and following the characteristics trajectory, we deduce
\Be \label{first cov first part}
 x^1 + \int^{t}_{t - \tb}  V(s;t,x, v) \dd s + (m, n, 0) = x,
\Ee
\Be \label{first cov second part}
\vb +  \int^{t}_{t - \tb}  - \nabla \Phi ( X(s;t,x, v)) \dd s = v.
\Ee
Inputting \eqref{first cov second part} into \eqref{first cov first part}, we have
\Be \label{x12 expression}
\begin{split}
& x^1_1 + \tb v_1 + m = x_1,
\\& x^1_2 + \tb v_2 + n = x_2.
\end{split}
\Ee
From \eqref{first cov second part}, we obtain $\tb = \tb (v_3)$ and
\Be 
\begin{split}
& \frac{\partial \tb}{\partial v} 
= (0, 0, \frac{\partial \tb}{\partial v_3}), 
\ \text{and} \
\frac{\partial x^1_\pll}{\partial v} 
= 
\begin{pmatrix}
    - \tb & 0 & - v_1 \frac{\partial \tb}{\partial v_3}\\
    0 & - \tb & - v_2 \frac{\partial \tb}{\partial v_3}
\end{pmatrix}.
\end{split}
\Ee
Therefore, we get
\Be \label{v xbtb cov first expression}
    \det 
    \Big(
    \frac{\partial x^1_{\pll}}{\partial v}, \frac{\partial \tb}{\partial v}
    \Big)
    = (\tb)^2 \times \frac{\partial \tb}{\partial v}.
\Ee
Now recall \eqref{characteristics}, 
\Be \label{V3 characteristics}
\frac{\dd}{\dd s} V_3 (s; t, x, v) 
= - \frac{1}{(1 + X_3 (s; t, x, v)) \ln (a)}.
\Ee
From Lemma \ref{conservative field},
\Be 
\frac{|V_3 (s;t,x,v)|^2}{2} + \Phi(X(s;t,x,v)) = \frac{v^2_3}{2}.
\Ee
Thus, we obtain
\Be \label{1+x form}
1 + X_3 (s;t,x,v) 
= a^{\frac{1}{2} (v^2_3 - V^2_3 (s;t,x,v))},
\Ee
and
\Be \label{v3 conservative}
|v_{\mathbf{b}, 3}| = |v_3|.
\Ee
Inputting \eqref{1+x form} into \eqref{V3 characteristics}, we derive
\Be \label{tb v_3 diff expression}
a^{- \frac{1}{2} V^2_3 (s;t,x,v)} \dd V_3 (s; t, x, v) 
= - \frac{a^{- \frac{1}{2} v^2_3}}{\ln (a)} \dd s.
\Ee
Taking the integration towards time $s \in [t - \tb, t]$ on \eqref{tb v_3 diff expression}, we get
\Be \label{tb v_3 first expression}
2 \int^{v_3}_{0} a^{- \frac{1}{2} V^2_3 (s;t,x,v)} \dd V_3 (s; t, x, v) 
= - \frac{a^{- \frac{1}{2} v^2_3}}{\ln (a)} \tb.
\Ee

{\color{black}
Note that $v_3 = V_3 (t;t,x,v)$ and $\vbn = V_3 (t - \tb;t,x,v)$.
Taking the integration towards time $s \in [t - \tb, t]$ on \eqref{tb v_3 diff expression}, we get
\Be \notag
\int^{v_3}_{\vbn} a^{- \frac{1}{2} V^2_3 (s;t,x,v)} \dd V_3 (s; t, x, v) 
= \int^{t}_{t - \tb} - \frac{a^{- \frac{1}{2} v^2_3}}{\ln (a)} \dd s
= - \frac{a^{- \frac{1}{2} v^2_3}}{\ln (a)} \tb.
\Ee
From Lemma \ref{conservative field}, $v_{\mathbf{b}, 3} = - v_3 > 0$. Further, since $a^{- \frac{1}{2} V^2}$ is an even function, we have
\Be \notag
\int^{v_3}_{\vbn} a^{- \frac{1}{2} V^2_3 (s;t,x,v)} \dd V_3 (s; t, x, v) 
= 2 \int^{v_3}_{0} a^{- \frac{1}{2} V^2_3 (s;t,x,v)} \dd V_3 (s; t, x, v) 
= - \frac{a^{- \frac{1}{2} v^2_3}}{\ln (a)} \tb.
\Ee
}

We estimate the following integration: 
\Be
\begin{split} 
\int^{|v_3|}_{0} a^{- \frac{1}{2} y^2} \dd y
\leq \int^{\infty}_{0} a^{- \frac{1}{2} y^2} \dd y
\lesssim \sqrt{\frac{1}{\ln (a)}}.
\end{split}
\Ee
On the other hand, 
\Be
\begin{split} 
\int^{|v_3|}_{0} a^{- \frac{1}{2} y^2} \dd y
\gtrsim \sqrt{\int^{\frac{\pi}{2}}_{0} \int^{v_3}_{0} a^{- \frac{1}{2} r^2} r \dd r \dd \theta}
\gtrsim \sqrt{\frac{1 - a^{- \frac{1}{2} v^2_3}}{\ln (a)}}.
\end{split}
\Ee
From \eqref{tb v_3 first expression}, we get
\Be \label{tb first expression}
\tb = 2 \ln (a) a^{\frac{1}{2} v^2_3}
\int^{|v_3|}_{0} a^{- \frac{1}{2} V^2_3 (s;t,x,v)} \dd V_3 (s; t, x, v).
\Ee
{\color{black}
Then
\Be \notag
\frac{2 \ln (a)}{\sqrt{\ln (a)}}
a^{\frac{1}{2} v^2_3} \sqrt{1 - a^{- \frac{1}{2} v^2_3}}
\lesssim \tb
\lesssim \frac{2 \ln (a)}{\sqrt{\ln (a)}} a^{\frac{1}{2} v^2_3}.
\Ee
Since $a$ is fixed, for simplicity we rewrite the above as
}
\Be \label{tb estimate}
a^{\frac{1}{2} v^2_3} \sqrt{1 - a^{- \frac{1}{2} v^2_3}} 
\lesssim \tb
\lesssim a^{\frac{1}{2} v^2_3}.
\Ee
Note that for $0 \leq |v_3| \ll 1$, we use the Taylor expansion on $a^{- \frac{1}{2} v^2_3}$, and obtain
\Be \label{small tb estimate}
\tb \gtrsim a^{\frac{1}{2} v^2_3} \sqrt{1 - a^{- \frac{1}{2} v^2_3}} 
\gtrsim \sqrt{1 - a^{- \frac{1}{2} v^2_3}} 
\gtrsim |v_3|. 
\Ee

Next, we take the derivative $\frac{\dd}{\dd v_3}$ on \eqref{tb first expression}, and write $\frac{\dd \tb}{\dd v_3}$ as
\Be 
\begin{split}
\frac{\dd \tb}{\dd v_3} 
= - 2 \ln (a) + v_3 \tb < 0.
\end{split}
\Ee
Thus, we derive that
\Be \label{tb v_3 diff second expression}
\begin{split}
1 + |v_3| a^{\frac{1}{2} v^2_3} \sqrt{1 - a^{- \frac{1}{2} v^2_3}}
\lesssim \Big| \frac{\dd \tb}{\dd v_3} \Big|
\lesssim 1 + |v_3| a^{\frac{1}{2} v^2_3}.
\end{split}
\Ee
{\color{black}
Since $a$ is a fixed constant, we can write
\Be \notag
\begin{split}
1 + |v_3| a^{\frac{1}{2} v^2_3} \sqrt{1 - a^{- \frac{1}{2} v^2_3}}
\lesssim \Big| \frac{\dd \tb}{\dd v_3} \Big|
\lesssim 1 + |v_3| \tb.
\end{split}
\Ee
}
Inputting \eqref{tb estimate}, \eqref{tb v_3 diff second expression} into \eqref{v xbtb cov first expression}, we get the following:
\Be 
 \Big|  \det \Big(
    \frac{\partial x^1_{\pll}}{\partial v}, \frac{\partial \tb}{\partial v}
    \Big) \Big|
= (\tb)^2 \times \Big| \frac{\partial \tb}{\partial v} \Big|
\gtrsim (\tb)^2 \times \Big( 1 + |v_3| a^{\frac{1}{2} v^2_3} \sqrt{1 - a^{- \frac{1}{2} v^2_3}} \Big) 
\gtrsim (\tb)^2,
\Ee
and
\Be 
\begin{split}
 \Big|  \det \Big(
    \frac{\partial x^1_{\pll}}{\partial v}, \frac{\partial \tb}{\partial v_3}
    \Big) \Big|
= (\tb)^2 \times \Big| \frac{\partial \tb}{\partial v} \Big|
\lesssim (\tb)^2 \times \big( 1 + |v_3| \tb \big) 
\end{split}
\Ee
Therefore, we conclude 
\Be
\frac{1}{\tb^2 (1 + |v_3| \tb) }
\lesssim \Big|  \det \Big(
    \frac{\partial x^1_{\pll}}{\partial v}, \frac{\partial \tb}{\partial v}
    \Big) \Big|^{-1}
\lesssim \frac{1}{\tb^2},
\Ee
and we conclude \eqref{jacob:mapV}.
\end{proof}

\begin{remark}
\label{rmk:vjmn}
We can apply Proposition \ref{prop:mapV} on 
$v^{j} \in \mathcal{V}_{j} 
\mapsto
(x^{j+1}, \tb^{j})
:=
(x_{\mathbf{b}} (x^{j}, v^{j}), \tb (x^{j}, v^{j}))$, and this is also a local bijective mapping. For fixed $\tb^{j}$ and $x^{j+1}$, we introduce the set of velocities $\{v^{m, n}_{j}\}$ with $m, n \in \mathbb{Z}$ such that
\Be \label{def:vjmn}
v^{m, n}_{j} 
\in \mathcal{V}_{j} 
 \mapsto
(x^{j+1} + (m, n, 0), \tb^{j})
= (x^{j+1}, \tb^{j})
\in \partial \Omega \times 
[0, t_{j}],
\ee
with the change of variable formula as
\be \label{vjmn}
\dd v^{m, n}_{j} \lesssim |\tb^{j}| ^{-2} \dd \tb^{j} \dd S_{x^{j+1}}.
\Ee 
\end{remark}


{\color{black}
Because of the periodic domain, we will gain an infinite sum of maxwellian terms as the integrand after the change of variable in Remark \ref{rmk:vjmn}. In the following lemma, we do an estimate on this infinite sum.
}

\begin{lemma} \label{lem: sum of mu}
Consider $(X,V)$ solving \eqref{characteristics} with $x^{i-1} \in \p\O$, for $\tb^{i-2} \geq 1$,
\Be \label{t>1,i-2}
\sum\limits_{m, n \in \mathbb{Z}} \mu (x^{i-1}, v^{m, n}_{i-2, \mathbf{b}}) 
\lesssim |\tb^{i-2}|^{4 - \A}.
\Ee
For $0 \leq \tb^{i-2} < 1$,
\Be \label{t<1,i-2}
\sum\limits_{m, n \in \mathbb{Z}} \mu (x^{i-1}, v^{m, n}_{i-2, \mathbf{b}})
\lesssim \sum\limits_{|m| < 2, |n| < 2} \mu (x^{i-1}, v^{m, n}_{i-2, \mathbf{b}}) +
e^{-\frac{1}{2 (\tb^{i-2})^2}},
\Ee
where $v^{m, n}_{i-2, \mathbf{b}} = \vb (x^{i-1}, v^{m, n}_{i-2})$, which was defined in \eqref{def:t_k} and \eqref{def:vmn}.
\end{lemma} 

\begin{proof}

Here, for the sake of simplicity, we have abused the notations temporarily: 
$$
x^i = (x^i_1, x^i_2, x^i_3)= (x^i_\parallel, x^i_3), \ 
v^{m, n}_{i, \mathbf{b}} = (v^{m, n}_{i, \mathbf{b}_1}, v^{m, n}_{i, \mathbf{b}_2}, v^{m, n}_{i, \mathbf{b}_3}) = (v^{m, n}_{i, \mathbf{b}_\parallel}, v^{m, n}_{i, \mathbf{b}_3}).
$$
To estimate $v^{m, n}_{i-2, \mathbf{b}_\parallel}$, we recall \eqref{x12 expression}, and get
\Be \label{vmnb12 first estimate}
\begin{split}
& |v^{m, n}_{i-2, \mathbf{b}_1}| 
= \frac{|x^{i-1}_1 + m - x^{i-2}_1|}{\tb^{i-2}},
\ \ |v^{m, n}_{i-2, \mathbf{b}_2}| = \frac{|x^{i-1}_2 + n - x^{i-2}_2|}{\tb^{i-2}}.
\end{split}
\Ee
Now we split the length of $\tb^{i-2}$ into two cases: 

\medskip

\textbf{Case 1:} $\tb^{i-2} \geq 1$. From \eqref{vmnb12 first estimate}, for $|m| \geq (\tb^{i-2})^2$, we bound 
\Be \notag
\frac{|x^{i-1}_1 + m - x^{i-2}_1|}{\tb^{i-2}}
\gtrsim \frac{|m|}{2 \tb^{i-2}}.
\Ee
Similarly, for $|n| \geq (\tb^{i-2})^2$, we bound
\Be \notag
\frac{|x^{i-1}_2 + n - x^{i-2}_2|}{\tb^{i-2}}
\gtrsim \frac{|n|}{2 \tb^{i-2}}.
\Ee
For $|m| < (\tb^{i-2})^2$, we bound $|v^{m, n}_{i-2, \mathbf{b}_1}| \geq 0$, and for $|n| < (\tb^{i-2})^2$, we bound $|v^{m, n}_{i-2, \mathbf{b}_2}| \geq 0$. In order to derive \eqref{t>1,i-2}, we divide $\{v^{m, n}_{i-2, \mathbf{b}}\}_{m, n \in \mathbb{Z}}$ into four parts.

\begin{enumerate}
\item[(a)] For $|m| < (\tb^{i-2})^2$ and $|n| < (\tb^{i-2})^2$, we bound
$a^{\frac{1}{2} |v^{m, n}_{i-2, \mathbf{b}_3}|^2} 
\gtrsim \tb^{i-2}$ in \eqref{tb estimate}. Therefore, we have
\Be \label{t>1m<n<}
\sum\limits_{|m| < (\tb^{i-2})^2, |n| < (\tb^{i-2})^2} \mu (x^{i-1}, v^{m, n}_{i-2, \mathbf{b}}) 
\lesssim (\tb^{i-2})^4 e^{-\frac{1}{2} |v^{m, n}_{i-2, \mathbf{b}}|^2}
\lesssim (\tb^{i-2})^4 (\tb^{i-2})^{- \A} 
= |\tb^{i-2}|^{4 - \A}.
\Ee

\item[(b)] For $|m| < (\tb^{i-2})^2$ and $|n| \geq (\tb^{i-2})^2$, we bound
$|v^{m, n}_{i-2, \mathbf{b}_2}| \gtrsim \frac{|n|}{2 \tb^{i-2}}$. Thus, we have 
\Be \label{t>1m<n>}
\begin{split}
 \sum\limits_{|m| < (\tb^{i-2})^2, |n| \geq (\tb^{i-2})^2} \mu (x^{i-1}, v^{m, n}_{i-2, \mathbf{b}})
& \lesssim \sum\limits_{|m| < (\tb^{i-2})^2} \mu (x^{i-1}, (0, \frac{|n|}{2 \tb^{i-2}}, v^{m, n}_{i-2, \mathbf{b}_3}))
\\& \lesssim (\tb^{i-2})^2 (\tb^{i-2})^{- \A}  \sum\limits^{\infty}_{n=0} \mu (x^{i-1}, \frac{|n|}{2 \tb^{i-2}})
\\& \leq (\tb^{i-2})^2 (\tb^{i-2})^{- \A}  
(1 - e^{-\frac{1}{8 (\tb^{i-2})^2}})^{-1} 
\\& \lesssim (\tb^{i-2})^4 (\tb^{i-2})^{- \A} 
= |\tb^{i-2}|^{4 - \A},
\end{split}
\Ee
where the last inequality holds from the Taylor expansion. 

\item[(c)] For $|m| \geq (\tb^{i-2})^2$ and $|n| < (\tb^{i-2})^2$ case, we bound
$|v^{m, n}_{i-2, \mathbf{b}_1}| \gtrsim \frac{|m|}{2 \tb^{i-2}}$. Similar as in \eqref{t>1m<n>}, we get
\Be \label{t>1m>n<}
\sum\limits_{|m| \geq (\tb^{i-2})^2, |n| < (\tb^{i-2})^2} \mu (x^{i-1}, v^{m, n}_{i-2, \mathbf{b}}) \lesssim |\tb^{i-2}|^{4 - \A}.
\Ee

\item[(d)] For $|m| \geq (\tb^{i-2})^2$ and $|n| \geq (\tb^{i-2})^2$, we use two lower bounds
$|v^{m, n}_{i-2, \mathbf{b}_2}| \gtrsim \frac{|n|}{2 \tb^{i-2}}$, $|v^{m, n}_{i-2, \mathbf{b}_1}| \gtrsim \frac{|m|}{2 \tb^{i-2}}$. Then, we derive that
\Be \label{t>1m>n>}
\begin{split}
\sum\limits_{|m| \geq (\tb^{i-2})^2, |n| \geq (\tb^{i-2})^2} \mu (x^{i-1}, v^{m, n}_{i-2, \mathbf{b}}) 
& \lesssim \sum\limits^{\infty}_{m, n = 0} \mu (x^{i-1}, ( \frac{|m|}{2 \tb^{i-2}}, \frac{|n|}{2 \tb^{i-2}}, v^{m, n}_{i-2, \mathbf{b}_3}))
\\& \lesssim (\tb^{i-2})^{- \A}  \sum\limits^{\infty}_{n=0} \mu (x^{i-1}, \frac{|n|}{2 \tb^{i-2}})
(1 - e^{-\frac{1}{8 (\tb^{i-2})^2}})^{-1} 
\\& \lesssim (\tb^{i-2})^{- \A} 
(1 - e^{-\frac{1}{8 (\tb^{i-2})^2}})^{-2} 
\lesssim (\tb^{i-2})^4 (\tb^{i-2})^{- \A} 
= |\tb^{i-2}|^{4 - \A}.
\end{split}
\Ee
\end{enumerate}
From \eqref{t>1m<n<}, \eqref{t>1m<n>}, \eqref{t>1m>n<} and \eqref{t>1m>n>}, we conclude that for $\tb^{i-2} \geq 1$,
\Be \notag
\sum\limits_{m, n \in \mathbb{Z}} \mu (x^{i-1}, v^{m, n}_{i-2, \mathbf{b}}) 
\lesssim |\tb^{i-2}|^{4 - \A}.
\Ee

\medskip

\textbf{Case 2:} $0 \leq \tb^{i-2} < 1$.
In this case $\tb^{i-2}$ is small, for $|m| \geq 2$ and $|n| \geq 2$, we bound \eqref{vmnb12 first estimate} as 
\Be \notag
\begin{split}
& \frac{|x^{i-1}_1 + m - x^{i-2}_1|}{\tb^{i-2}} \gtrsim \frac{|m|}{2 \tb^{i-2}}, \ \ 
\frac{|x^{i-1}_2 + n - x^{i-2}_2|}{\tb^{i-2}}
\gtrsim \frac{|n|}{2 \tb^{i-2}}.
\end{split}
\Ee
For $|m| < 2$, we bound $|v^{m, n}_{i-2, \mathbf{b}_1}| \geq 0$, and for $|n| < 2$, we bound $|v^{m, n}_{i-2, \mathbf{b}_2}| \geq 0$.
To obtain \eqref{t<1,i-2}, we again divide $\{v^{m, n}_{i-2, \mathbf{b}}\}_{m, n \in \mathbb{Z}}$ into four parts.

\begin{enumerate}
\item[(a)] For $|m| < 2$ and $|n| < 2$, we keep the following five terms summation: 
\Be \label{t<1m<n<}
\sum\limits_{|m| < 2, |n| < 2} \mu (x^{i-1}, v^{m, n}_{i-2, \mathbf{b}}).
\Ee

\item[(b)] For $|m| < 2$ and $|n| \geq 2$, we bound
$|v^{m, n}_{i-2, \mathbf{b}_2}| \gtrsim \frac{|n|}{2 \tb^{i-2}}$. Thus, we have 
\Be \label{t<1m<n>}
\begin{split}
\sum\limits_{|m| < 2, |n| \geq 2} \mu (x^{i-1}, v^{m, n}_{i-2, \mathbf{b}})
& \lesssim \sum\limits_{|m| < 2, |n| \geq 2} \mu (x^{i-1}, (0, \frac{|n|}{2 \tb^{i-2}}, v^{m, n}_{i-2, \mathbf{b}_3}))
\\& \lesssim \sum\limits^{\infty}_{n=2} \mu (x^{i-1}, \frac{|n|}{2 \tb^{i-2}})
\lesssim \sum\limits^{\infty}_{n = 2} e^{-\frac{n^2}{8 (\tb^{i-2})^2}}
\\& \leq e^{-\frac{1}{2 (\tb^{i-2})^2}}
(1 - e^{-\frac{1}{8 (\tb^{i-2})^2}})^{-1} 
\lesssim e^{-\frac{1}{2 (\tb^{i-2})^2}},
\end{split}
\Ee
where the last inequality holds from $0 \leq \tb^{i-2} < 1$. 

\item[(c)] For $|m| \geq 2$ and $|n| < 2$ case, we bound
$|v^{m, n}_{i-2, \mathbf{b}_1}| \gtrsim \frac{|m|}{2 \tb^{i-2}}$. Similar as in \eqref{t<1m<n>}, we get
\Be \label{t<1m>n<}
\sum\limits_{|m| \geq 2, |n| < 2} \mu (x^{i-1}, v^{m, n}_{i-2, \mathbf{b}}) 
\lesssim e^{-\frac{1}{2 (\tb^{i-2})^2}}.
\Ee

\item[(d)] For $|m| \geq 2$ and $|n| \geq 2$, we bound
$|v^{m, n}_{i-2, \mathbf{b}_2}| \gtrsim \frac{|n|}{2 \tb^{i-2}}$, $|v^{m, n}_{i-2, \mathbf{b}_1}| \gtrsim \frac{|m|}{2 \tb^{i-2}}$ and we derive
\Be \label{t<1m>n>}
\begin{split}
\sum\limits_{|m| \geq 2, |n| \geq 2} \mu (x^{i-1}, v^{m, n}_{i-2, \mathbf{b}}) 
& \lesssim \sum\limits^{\infty}_{m, n = 2} \mu (x^{i-1}, ( \frac{|m|}{2 \tb^{i-2}}, \frac{|n|}{2 \tb^{i-2}}, v^{m, n}_{i-2, \mathbf{b}_3}))
\\& \lesssim \sum\limits^{\infty}_{n=2} \mu (x^{i-1}, \frac{|n|}{2 \tb^{i-2}})
e^{-\frac{1}{2 (\tb^{i-2})^2}}
\lesssim e^{-\frac{1}{2 (\tb^{i-2})^2}}.
\end{split}
\Ee
\end{enumerate}
From \eqref{t<1m<n<}, \eqref{t<1m<n>}, \eqref{t<1m>n<} and \eqref{t<1m>n>}, we conclude that for $0 \leq \tb^{i-2} < 1$,
\Be \notag
\sum\limits_{m, n \in \mathbb{Z}} \mu (x^{i-1}, v^{m, n}_{i-2, \mathbf{b}})
\lesssim 
\sum\limits_{|m| < 2, |n| < 2} \mu (x^{i-1}, v^{m, n}_{i-2, \mathbf{b}}) +
e^{-\frac{1}{2 (\tb^{i-2})^2}},
\Ee
so we prove \eqref{t>1,i-2} and \eqref{t<1,i-2}.
\end{proof}


\section{Weighted \texorpdfstring{$L^1$}{L1}-Estimates}
\label{sec: L1 estimate}

The main purpose of this section is to prove Theorem \ref{theorem_1}, in which we do $L^1$-estimates on fluctuations. Then we show the existence and uniqueness of the stationary solution.


\subsection{\texorpdfstring{$f (t, x, v)$}{f(t,x,v)} via Stochastic Cycles}

{\color{black}
The main propose of this section is to show Lemma \ref{lemma:energy}, where we control $\| \varphi(\tf) f(t) \|_{L^1_{x,v}}$ under some weight function $\varphi$.
To prove Lemma \ref{lemma:energy}, we first express $f (t, x, v)$ with the stochastic cycles in Lemma \ref{sto_cycle}, then we do some energy estimates in Lemma \ref{lem:energy estimate on f}.
}

\begin{lemma} \label{sto_cycle}

{\color{black}
For any integer $k \geq 2$, suppose
}
$f (t, x, v)$ solves \eqref{equation for F} and \eqref{diff_F}, and $t_* \leq t$, then we have
\begin{align}
    f (t, x, v) = & \mathbf{1}_{t^1 < t_*}
    f (t_*, X(t_*; t, x, v), V(t_*; t, x, v)) \label{expand_h1} 
    \\&  + \mu (x^1, \vb) \sum\limits^{k-1}_{i=1}  \int_{\prod_{j=1}^{i} \mathcal{V}_j} \Big\{   \mathbf{1}_{t^{i+1}<t_* \leq t^{i }} f (t_*, X(t_*; t^i, x^i, v^i), V(t_*; t^i, x^i, v^i))  \Big\}
      \dd  \Sigma_{i}
\label{expand_h2}
    \\& + \mu (x^1, \vb) \int_{\prod_{j=1}^{k } \mathcal{V}_j}   
    \mathbf{1}_{t^{k } \geq t_* }
    f (t^{k}, x^{k }, v^{k })
     \dd  \Sigma_{k},
\label{expand_h3}
\end{align} 
where 
$\dd  {\Sigma}_{i} 
:= \frac{ \dd \sigma_{i}}{ \mu (x^{i+1}, v^{i})} \dd \sigma_{i-1} \cdots  \dd \sigma_1$, with $\dd \sigma_j = \mu (x^{j+1}, v^{j}) \{ n(x^j) \cdot v^j \} \dd v^j$ in \eqref{def:sigma measure}. Here, $(X,V)$ solves \eqref{characteristics}.

\end{lemma}

\begin{proof}
Following the similar steps in Lemma \ref{sto_cycle_1} and Remark \ref{probability}, we can obtain this Lemma.
\end{proof}

\begin{lemma} \label{lemma:energy}


{\color{black}
Given a function $\varphi: [0, \infty) \to \R$, suppose $\varphi$ satisfies that for any $\tau \geq 0$, $\varphi (\tau ) \geq0$, $\varphi^\prime \geq0$, and
}
\Be \label{cond:varphi} 
\int_1^\infty \tau^{3 - \A} \varphi(\tau) \dd \tau < \infty. 
\Ee
Suppose $f$ solves \eqref{equation for F} and \eqref{diff_F}, there exists $C>0$ independent of $t_*$, $t$, such that for all $0 \leq t_* \leq t$, 
\Be \label{energy_varphi}
\begin{split}
& \ \ \ \ \| \varphi(\tf) f(t) \|_{L^1_{x,v}}
  + \int^{t}_{t_*}
  \| \varphi^\prime(\tf) f  \|_{L^1_{x,v}} \dd s
  +  \int_{t_*}^{t} | \varphi(\tf) f|_{L^1_{\gamma_+}} \dd s
\\& \leq \| \varphi(\tf) f(t_*) \|_{L^1_{x,v}} +
 C (t - t_* + 1) \| f(t_*) \|_{L^1_{x,v}} + \frac{1}{4}  \int^{t}_{t_*} |f |_{L^1_{\gamma_+}} \dd s.
\end{split}
\Ee 
\end{lemma} 


{\color{black}
For the proof of Lemma \ref{lemma:energy}, we shall start it from the energy estimate, Lemma \ref{lem:energy estimate on f}.
}

\begin{lemma}[\cite{JK2}]
\label{lem:energy estimate on f}
Suppose $f$ solves \eqref{equation for F} and \eqref{diff_F}, then for $0 \leq t_* \leq t$ with $0 < \delta < \min (1, t-t_*)$,
\begin{align}
& \ \ \ \ \| f(t) \|_{L^1_{x,v}}  \leq \| f(t_*) \|_{L^1_{x,v}}, \label{maximum} 
\\& \int^{t}_{t_*}  |f(s )|_{L^1_{\gamma_+}} \dd s 
 \leq \Big \lceil \frac{t-t_*}{\delta}   \Big \rceil \| f(t_*) \|_{L^1_{x,v}}
+ O(\delta^2) \int^{t}_{t_*}  |f(s )|_{L^1_{\gamma_+}} \dd s, \label{trace}
\end{align}
and if $f_0$ is non-negative, so is $f(t, x, v)$ for all $(t, x, v) \in \R_{+} \times \Omega \times \R^3$.
\end{lemma}

\begin{proof}

{\color{black}
Since $f (t, x, v)$ solves \eqref{equation for F} and \eqref{diff_F} in the $L^1$ sense, according to \cite[Lemma 2]{B}, $|f (t, x, v)|$ is also a solution to \eqref{equation for F} and \eqref{diff_F}.
}

From \eqref{equation for F} and \eqref{diff_F}, taking integration 
{\color{black} 
on $|f(t,x,v)|$
}
over $(t_*, t) \times \Omega \times \R^3$, 
we derive that
\Be \notag
\| f(t) \|_{L^1_{x,v}} + \int^t_{t*} \int_{\gamma_+} |f| \dd s - \int^t_{t*} \int_{\gamma_-}|f| \dd s = \| f(t_*) \|_{L^1_{x,v}}.
\Ee
Due to the choice of $\mu (x, v)$ in \eqref{diff_F}, for $\forall t > 0$,
\Be \notag
\int_{\gamma_-} |f (t, x, v)| |n(x) \cdot v| \dd S_x \dd v 
= \Big|\int_{\gamma_+} f (t, x, v) \{n(x) \cdot v\} \dd S_x \dd v 
\Big|.
\Ee
Therefore, we have
\Be \notag
\begin{split}
& \ \ \ \ \int^t_{t*} \int_{\gamma_+}|f| \dd s - \int^t_{t*} \int_{\gamma_-}|f| \dd s
\\& = \int^t_{t*} \int_{\gamma_+}|f| \dd s -\int^t_{t*} \Big|\int_{\gamma_+}f\Big| \dd s
\geq  \int^t_{t*} \int_{\gamma_+}|f| \dd s -  \int^t_{t*} \int_{\gamma_+}|f| \dd s =0,
\end{split}
\Ee 
therefore we prove \eqref{maximum}.

Next we work on \eqref{trace}. For $\delta \in (0, t-t_*)$ and $(x,v) \in \gamma_+$,  
{\color{black}
we split the time interval $[t^*, t]$ into some subintervals as follows:
\[
[t^*, t^* + \delta], [t^* + \delta, t^* + 2 \delta], \ldots, [t^* + \lceil \frac{t-t_*-\delta}{\delta} \rceil \delta, t].
\]
Since $f$ is invariant along the characteristic, we backward $f (s,x,v)$ into a new time depending on $s$ and $\tb(x,v)$. Then we do estimates on different cases. 
}
\Be \label{|f|1}
\begin{split}
|f(s,x,v)|
& \leq \underbrace{ \sum\limits^{ \lceil \frac{t-t_*-\delta}{\delta} \rceil }_{k=1} \mathbf{1}_{t_* + k \delta \leq s \leq t_* + (k+1) \delta, \ \delta < \tb(x,v)}
|f(t_* + k \delta, X(t_* + k \delta, s, x, v), V(t_* + k \delta, s, x, v))| }_{\eqref{|f|1}_1} 
\\& \ \ \ \ + \underbrace{ \mathbf{1}_{t_* + \delta \leq s, \ \delta \geq \tb(x,v)}
|f(s- \tb (x,v), \xb (x,v), \vb (x,v))|}_{\eqref{|f|1}_2}
\\& \ \ \ \ + \underbrace{ \mathbf{1}_{s - t_* < \delta, \ s - t_* < \tb(x,v)}
|f(t_*, X(t_*, s, x, v), V(t_*, s, x, v))| }_{\eqref{|f|1}_3} 
\\& \ \ \ \ + \underbrace{ \mathbf{1}_{\tb(x,v) \leq s - t_* < \delta}
|f(s- \tb (x,v), \xb (x,v), \vb (x,v))|}_{\eqref{|f|1}_4},
\end{split} 
\Ee
where $s \in [t_*,t]$. 

First we do estimate on $\eqref{|f|1}_1$.
From \eqref{COV}, \eqref{maximum} and $t_* + \delta \leq s \leq t$ with $\delta < \tb (x, v)$,
\begin{align} 
& \ \ \ \ \int^{t}_{t_*} \int_{\gamma_+} \eqref{|f|1}_1 \dd s \notag
\\& \leq \sum\limits^{ \lceil \frac{t-t_*-\delta}{\delta} \rceil }_{k=1} \int_{\gamma_+}
\int^{ t_* + k \delta + \tb (x, v)}_{ t_* + k \delta}
|f(t_* + k \delta, X(t_* + k \delta, s, x, v), V(t_* + k \delta, s, x, v))| \dd s 
\{n(x) \cdot v\} \dd S_x \dd v \notag 
\\& \leq \sum\limits^{ \lceil \frac{t-t_*-\delta}{\delta} \rceil  }_{k=1} \| f(t_* + k \delta) \|_{L^1_{x,v}}  
\leq \Big \lceil \frac{t-t_*-\delta}{\delta}   \Big \rceil \times \| f(t_*) \|_{L^1_{x,v}}.
\end{align}

Now we consider $\eqref{|f|1}_2$. For $y= \xb(x,v)$ and $s \in [t_*,t]$, we have 
\Be \label{delta tf}
\mathbf{1}_{\delta \geq \tb(x,v)}
= \mathbf{1}_{\delta \geq \tf(y, \vb)}.
\Ee
From \eqref{tb estimate} and \eqref{small tb estimate}, for $(x,v) \in \gamma_{+}$ and $v_3 \lesssim \tb (x, v) < \delta < 1$. Then we get 
\be \notag
\mathbf{1}_{ |v_3| \lesssim \delta} \geq \mathbf{1} _{\delta> \tb (x,v)}.
\ee
{\color{black}
Thus, we compute that
}
\Be \label{estimate on delta}
\begin{split} 
& \ \ \ \ \int_{n(x) \cdot v > 0}
\mathbf{1} _{\delta> \tb (x, v)} \mu (\xb, v) |n(x) \cdot v|\dd v
\\& \lesssim \int_{ |v_3| \lesssim \delta} \mu (\xb, v) | n(x) \cdot v | \dd v  
\\& \leq \int_{ |v_3| \lesssim \delta} e^{- \frac{v_3^2}{2}} |v_3| \dd v_3
\lesssim C \delta^2.
\end{split}
\Ee
From \eqref{COV_bdry}, \eqref{delta tf} and using the Fubini's theorem, we derive
\begin{align} 
\int^{t}_{t_*} \int_{\gamma_+} \eqref{|f|1}_2 \dd s
& = \int_{\gamma_+}
\int^t_{ t_* + \delta}
\mathbf{1}_{\delta \geq \tb(x,v)}
|f(s- \tb(x, v), \xb, \vb)| \dd s 
\{n(x) \cdot v\} \dd S_x \dd v \notag \\
& \leq \int_{\p\O} \int_{n(y) \cdot v<0}
\mathbf{1} _{\delta> \tf(y, v)}
\int^t_{t_*}|f(s, y, v)| \dd s 
|n(y) \cdot v| \dd S_{y} \dd v \notag \\
& \leq \int_{\p\O}
\underbrace{\Big( \int_{n(y) \cdot v<0}
\mathbf{1} _{\delta> \tf(y, v)} \mu (y, v) |n(y) \cdot v|\dd v\Big)}_{\eqref{est:|f|2}_*}
\int^t_{t_*}
\int_{n(y) \cdot v^1>0} 
|f(s,y, v^1)| \{n(y) \cdot v^1\} \dd v^1
 \dd s 
 \dd S_{y} .
 \label{est:|f|2}
\end{align}
From \eqref{estimate on delta}, we derive
\be \notag
\eqref{est:|f|2} \leq O(\delta^2) \int^t_{t_*} \int_{\gamma_+} |f| \dd s.
\ee 
From \eqref{COV}
and $s < t_* + \tb(x,v)$, we have
\Be \notag
\int^{t}_{t_*} \int_{\gamma_+} \eqref{|f|1}_3 \dd s
\leq \int^{t_* + \tb(x,v)}_{t_*} \int_{\gamma_+} \eqref{|f|1}_3 \dd s 
\leq \| f(t_*) \|_{L^1_{x,v}}.
\Ee
Again setting $y= \xb(x,v)$ and $s \in [t_*,t]$, we have 
\Be \label{delta tf 2}
\mathbf{1}_{\tb(x,v) \leq s - t_* < \delta}
\leq \mathbf{1}_{\delta \geq \tf(y, \vb)}.
\Ee
From \eqref{COV_bdry}, \eqref{delta tf 2} and using the Fubini's theorem, we derive
\begin{align} 
\int^{t}_{t_*} \int_{\gamma_+} \eqref{|f|1}_4 \dd s
& = \int_{\gamma_+}
\int^{t_* + \delta}_{ t_* + \tb(x,v)}
\mathbf{1}_{\delta \geq \tb(x,v)}
|f(s- \tb(x, v), \xb, \vb)| \dd s 
\{n(x) \cdot v\} \dd S_x \dd v \notag \\
& \leq \int_{\p\O} \int_{n(y) \cdot v<0}
\mathbf{1} _{\delta> \tf(y, v)}
\int^{t_* + \delta}_{ t_*} |f(s, y, v)| \dd s |n(y) \cdot v| \dd S_{y} \dd v \notag \\
&\leq \int_{\p\O}
\underbrace{\Big( \int_{n(y) \cdot v<0}
\mathbf{1} _{\delta> \tf(y, v)} \mu (y, v) |n(y) \cdot v|\dd v\Big)}_{\eqref{est:|f|4}_*}
\int^{t_* + \delta}_{ t_*} \int_{\gamma_+} |f| \dd s.
 \label{est:|f|4}
\end{align}
Then we conclude $\eqref{est:|f|4} \leq O(\delta^2) \int^t_{t_*} \int_{\gamma_+} |f| \dd s$, therefore we prove \eqref{trace}.

To prove the positivity property, we write 
\[
f_{-} = \frac{|f| - f}{2}.
\]
{\color{black}
Since both $f (t, x, v)$ and $|f (t, x, v)|$ are solutions to \eqref{equation for F} and \eqref{diff_F}, it is clear that $f_{-}$ also solves \eqref{equation for F} and \eqref{diff_F}.
}
From \eqref{maximum} and the assumption $f_0 \geq 0$, we have
\Be \label{positivity}
\| f_{-} (t) \|_{L^1_{x,v}} 
= \big\| \frac{|f| (t) - f (t)}{2} \big\|_{L^1_{x,v}} 
= \Big\| \frac{ \big(|f| - f \big) (t)}{2} \Big\|_{L^1_{x,v}}
\leq \big\| \frac{|f_0| - f_{0} }{2} \big\|_{L^1_{x,v}} = 0,
\Ee
then we conclude $f_{-} (t, x, v) = 0$ on $\Omega \times \R^3$.
\end{proof}

Now we are ready to prove Lemma \ref{lemma:energy}, which will be used frequently in this paper.

\begin{proof}[\textbf{Proof of Lemma \ref{lemma:energy}}] 

Considering the characteristics trajectory
$\big(s, X(s; t, x, v), V(s; t, x, v) \big)$ that is determined by $(t, x, v)$, then we apply this characteristics 
{\color{black}
on $\tf (x, v)$. Recall that $\tf$ is timely independent because of the timely independent field $\Phi(x)$.
}
\Be \notag
\begin{split}
    -1 & = \frac{d}{d s} \tf ( X(s; t, x, v), V(s; t, x, v))
    \\& = \frac{\partial}{\partial X} \tf (X, V) \cdot \frac{d}{ds} X
    + \frac{\partial}{\partial V} \tf (X, V) \cdot \frac{d}{ds} V,
\end{split}
\Ee
By setting $s = t$, we have
\Be \notag
\begin{split}
 \frac{\partial}{\partial x} \tf (x, v) \cdot v
    + \frac{\partial}{\partial v} \tf (x, v) \cdot - \nabla \Phi (x) = -1.
\end{split}
\Ee
{\color{black}
On the other hand, since $f (t, x, v)$ solves \eqref{equation for F} and \eqref{diff_F}, then $|f (t, x, v)|$ also solves \eqref{equation for F} and \eqref{diff_F}, that is, $[\p_t + v\cdot \nabla_x - \nabla \Phi \cdot \nabla_{v}] |f| = 0$.
}
Then, in the sense of distribution 
\Be \label{function for varphi f}
[\p_t + v\cdot \nabla_x - \nabla \Phi \cdot \nabla_{v}] \big(\varphi(\tf) |f| \big) 
= \varphi^\prime(\tf) [\p_t + v\cdot \nabla_x - \nabla \Phi \cdot \nabla_{v}](\tf) |f|
= -\varphi^\prime(\tf)| f|. 
\Ee
From \eqref{equation for F}, \eqref{function for varphi f}, $\varphi (\tau ) \geq0$, $\varphi^\prime \geq0$ and taking integration over $(t_*, t) \times \Omega \times \R^3$, 
we derive
\begin{align}
& \ \ \ \ \| \varphi(\tf) f (t) \|_{L^1_{x,v}} + \int^t_{t_*} \| \varphi^\prime(\tf) f (s) \|_{L^1_{x,v}} \dd s
+ \int^t_{t_*} \int_{\gamma_+} \varphi(\tf) |f| \dd s \notag
\\&
\leq \| \varphi(\tf) f (t_*) \|_{L^1_{x,v}}
+ \int^t_{t_*} 
\int_{\p\O}
\int_{n(x) \cdot v<0}
 \varphi(\tf) |f| |n(x) \cdot v|
\dd v
\dd S_x \dd s \notag
\\& = \| \varphi(\tf) f (t_*) \|_{L^1_{x,v}}
+ \int^t_{t_*} 
\int_{\p\O} 
\int_{n(x) \cdot v<0}
 \varphi(\tf) \mu (x, v) |n(x) \cdot v| \dd v \dd S_x \dd s \notag
\\& \hspace{7cm} \times
\int_{n(x)\cdot v^1>0}| f(s,x,v^1)| 
\{n(x^1) \cdot v^1\}\dd v^1. \label{rho.f}
\end{align}

Now we prove the following claim: 
If \eqref{cond:varphi} holds, then 
\Be \label{claim in lemma 3}
\sup_{x \in \p\O} \int_{n(x) \cdot v<0}
 \varphi(\tf)(x,v) \mu (x, v) |n(x) \cdot v| \dd v\lesssim 1.
\Ee
We split $\int_{n(x) \cdot v<0} \varphi(\tf)(x,v) \mu (x, v) |n(x) \cdot v| \dd v$ into two parts: integration over the regimes of $\tf \leq 1$ and $\tf > 1$. 

For $\tf \leq 1$, 
{\color{black}
since $x \in \p\O$ and $n(x) \cdot v<0$, we consider $\xf (x, v), \vf (x, v)$ and get $\tb (\xf, \vf) = \tf (x, v)$. Using \eqref{v3 conservative}, together with \eqref{small tb estimate}, we have 
\[
|v_3| = |\vfn| \lesssim \tb (\xf, \vf).
\]
This shows that $|v_3| \lesssim \tf (x, v) \leq 1$. Combining with $\varphi^\prime \geq0$, we bound
}
\Be \label{int:rho_1}
\begin{split}
 \int_{n(x) \cdot v<0} \mathbf{1}_{\tf \leq 1} \
\varphi(\tf) \mu (x, v) |n(x) \cdot v| \dd v
\lesssim \varphi(1) \int_{\R^3} 
 e^{-|v|^2 / 2} \dd v \lesssim 1. 
\end{split} 
\Ee

{\color{black}
For
}
$\tf > 1$, applying \eqref{mapV_f} in Proposition \ref{prop:mapV} and Lemma \ref{lem: sum of mu},
{\color{black}
together with $|n (x) \cdot v| \lesssim \tf (x, v)$ from \eqref{tb estimate}, we obtain
}
{\color{black}
\Be \label{int:rho_2} 
\int_{n (x) \cdot v<0} \varphi(\tf) \mu (x, v) |n (x) \cdot v| \dd v
\lesssim \int_{\p\O} \int_{1}^{\infty}
\varphi(\tf) \sum\limits_{m, n \in \mathbb{Z}} \mu (\xf, \vf)  \frac{|\tf|}{|\tf|^2} \dd \tf \dd S_{\xf}. 
\Ee
}
{\color{black}
From Lemma \ref{lem: sum of mu} and \eqref{cond:varphi}, we derive that 
}
\Be \notag
\eqref{int:rho_2}   
\lesssim \int_{1}^\infty
\frac{\varphi(\tf)}{|\tf|} |\tf|^{4 - \A}  \dd \tf
\lesssim \int_{1}^\infty \varphi(\tf) |\tf|^{3 - \A} \dd \tf\lesssim 1.
\Ee
Combining the above bound with \eqref{int:rho_1}, we prove \eqref{claim in lemma 3}. 
Then picking sufficiently small $\delta$ in \eqref{trace} and using \eqref{claim in lemma 3}, we conclude \eqref{energy_varphi}, through, for $C>1$,   
\Be \notag
\begin{split}
\eqref{rho.f} 
& \lesssim \int^t_{t_*} 
\int_{\gamma_+} |f(s,x,v^1)| \{n(x^1) \cdot v^1\} \dd v^1 \dd S_x \dd s 
\\& \leq C (t - t_* + 1) \| f(t_*) \|_{L^1_{x,v}} + \frac{1}{4}  \int^{t}_{t_*} |f(s)|_{L^1(\gamma_+)}.
\end{split}
\Ee  
\end{proof} 

\subsection{Lower bound with the unreachable defect}
\label{sec:lower_bound}

In this section, we prove the Proposition \ref{prop:Doeblin 1} to obtain a lower bound with the unreachable defect. It is the key to control the fluctuations. 


 
\begin{proposition} \label{prop:Doeblin 1}
Suppose $f$ solves \eqref{equation for F} and \eqref{diff_F}. Assume $f_0(x,v) \geq 0$. For any $T_0\gg1$ and $N\in \mathbb{N}$ there exists $\mathfrak{m}(x,v)\geq 0$, which only depends on $\O$ and $T_0$ (see \eqref{def:m} for the precise form), such that 
\Be \label{est:Doeblin}
f(NT_0,x,v)\geq  
\mathfrak{m}(x,v) \Big\{
\iint_{\O \times \R^3}f((N-1)T_0,x,v) \dd v \dd x 
-  \iint_{\O \times \R^3} \mathbf{1}_{t_\mathbf{f}(x,v)\geq \frac{T_0}{4}} f((N-1)T_0,x,v)  \dd v \dd x 
\Big\}.
\Ee
\end{proposition} 

\begin{proof}
\textbf{Step 1.} 
From \eqref{positivity} the assumption $f_0 (x, v) \geq 0$, we have $f(t,x,v) \geq 0$. From \eqref{expand_h1}-\eqref{expand_h3} and setting $t = NT_0$, $t_* = (N-1)T_0$, $k=2$, we can derive that 
\be \label{Doeblin_1}
f(NT_0, x,v) 
\geq \mathbf{1}_{ \tb(x,v) \leq \frac{T_0}{4}} \mu (x^{1}, \vb) \int_{\mathcal{V}_1} 
 \int_{\mathcal{V}_2}  
 \mathbf{1}_{t^2 \geq (N-1)T_0}
 f(t^2, x^2, v^2)  \{n(x^2) \cdot v^2\} \dd v^2 \dd \sigma_1.
\ee 

Now we apply Proposition \ref{prop:mapV}  on $v^1 \in \mathcal{V}_1$ with \eqref{mapV} and \eqref{jacob:mapV}. 
In order to have the bijective mapping with \eqref{mapV}, we restrict the range of $\vb^1$ as 
\Be \label{def:x^2}
\mathcal{V}_{1,b} := \{ \vb^1 \in \R^3:  x^2 + \int^{t}_{t - \tb} \big( \vb +  \int^{s}_{t - \tb}  - \nabla \Phi ( X(\alpha; t^1,x^1, v^1)) \dd \alpha \big) \dd s = x^1 \}.
\Ee
This implies all characteristic trajectories $X(\alpha; t^1,x^1, v^1)$ between $x^1$ and $x^2$ under $\vb^1 \in \mathcal{V}_{1,b}$ don't cross the periodic boundary.

Therefore, we derive 
\Be \label{Doeblin_2}
\begin{split}
\eqref{Doeblin_1} 
& \gtrsim
\mathbf{1}_{ \tb(x,v) \leq \frac{T_0}{4}}  \mu (x^{1}, \vb) 
 \int_{0}^{T_0 - \tb (x, v)}
  \int_{\p\O} 
  \underbrace{ \frac{n(x^1) \cdot v^1}{(\tb^1)^2 (1 + |v^1_3| \tb^1)} \mu (x^2, \vb^1) }_{\eqref{Doeblin_2}_*}
\\& \ \ \ \ \ \ \times 
   \int_{n(x^2) \cdot v^2>0}
   f(t^2, x^2, v^2)  \{n(x^2) \cdot v^2\} \dd v^2
    \dd S_{x^2}\dd \tb^1,
\end{split}
\Ee
where $t^2 = NT_0 - \tb(x,v) - \tb^1$ and $n(x^1) \cdot v^1 = v^1_3$.
  
\medskip 

\textbf{Step 2.} In order to bound the integrand of the first line in \eqref{Doeblin_2}, we will further restrict integration regimes. Note that $x^1 = \xb(x,v)$ is given, $x^2$ is free variables and $t^2 \geq (N-1) T_0$.

Now we restrict the integral regimes of the variable $\tb^1$ as 
  
\Be \label{def:T+}
\mathfrak{T}^{T_0} :=
 \Big\{
 \tb^1 \in  [0, \infty):
 T_0- \tb(x,v) - \min\Big(\tb(x^2, v^2) ,
 \frac{T_0}{4}
\Big)
  \leq \tb^1 \leq 
 T_0- \tb(x,v)
 \Big\}.
\Ee
As a consequence of \eqref{def:T+} 
and $\tb(x,v) \leq \frac{T_0}{4}$ in \eqref{Doeblin_2}, we will derive  \eqref{min:tb1} and \eqref{cond:tf},
\Be \label{min:tb1}
\frac{T_0}{2} \leq T_0 - \tb(x,v) - \frac{T_0}{4}
\leq \tb^1 \leq T_0.
\Ee
Secondly, we prove \eqref{cond:tf}. Note that if $\tb^1 \in \mathfrak{T}^{T_0}$, we have
\Be \notag
(N-1)T_0
\leq t^2 =
 NT_0 - \tb(x,v) - \tb^1 \leq 
 (N-1)T_0 
 +\min \{\tb(x^2, v^2) ,
 \frac{T_0}{4}
\}.
\Ee
This implies that, for $y_* = X((N-1)T_0; t^2, x^2, v^2)$ and $v_* = V((N-1)T_0; t^2, x^2, v^2)$, we have 
\Be \label{cond:tf}
\tf(y_*, v_*)= t^2 - (N-1)T_0 = T_0 - \tb(x,v) - \tb^1 \in 
 \Big[0, \frac{T_0}{4}\Big],
\Ee
where we use $\tf (y_*, v_*) \leq \tb (x^2, v^2)$ since $x^2 = \xf(y_*, v_*)$. 
 
\medskip

\textbf{Step 3.} For \eqref{Doeblin_2}, we apply the restriction of integral regimes in \eqref{def:x^2} and \eqref{def:T+}.
{\color{black}
Note that $\frac{n(x^1) \cdot v^1}{(\tb^1)^2 (1 + |v^1_3| \tb^1)} = \frac{|v^1_3|}{(\tb^1)^2 + |v^1_3| (\tb^1)^3} \geq \frac{1}{|\tb^1|^3}$. Then
\Be \notag
\begin{split}
\eqref{Doeblin_2}_*
\gtrsim \frac{1}{|\tb^1|^3} \mu (x^2, \vb^1)
& = \frac{1}{|\tb^1|^3}  
\mu \big( |n(x^2) \cdot \vb^1| \big) \mu \big( \frac{|x^2 - x^1|}{|\tb^1|}  \big)
\\& \gtrsim \frac{1}{|\tb^1|^3}  
\mu (|v^1_3|) \mu \big( \frac{\sqrt{2}}{|\tb^1|}  \big)
\gtrsim \frac{1}{(T_0)^3} e^{- \frac{1}{2} |v^1_3|^2} e^{- \frac{4}{T^2_0}} 
\\& \gtrsim \frac{1}{(T_0)^3} (\tb^1)^{-(\A+1)}
\geq (T_0)^{-4-\A},
\end{split}
\Ee
where the second last inequality follows from $T_0 \gg 1$, \eqref{tb estimate} and $\A \leq \frac{1}{\ln (a)} < \A + 1$.
}

Finally, we get 
\begin{align}
\eqref{Doeblin_2} 
& \geq  \mathbf{1}_{ \tb(x,v) \leq \frac{T_0}{4}} (T_0)^{-4-\A}  
  \mu (x^1, \vb) \int_{\p\O} \dd S_{x^2} \int_{n(x^2) \cdot v^2 > 0}
 \dd v^2 \{n(x^2) \cdot v^2\} \notag 
\\&  \ \ \ \ \ \ \times 
	\int_{\mathfrak{T}^{T_0}} \dd \tb^1
   f(
   NT_0 -\tb(x,v) - \tb^1
   ,x^2, v^2)  \notag 
\\& \gtrsim  \mathbf{1}_{ \tb(x,v) \leq \frac{T_0}{4}} (T_0)^{-4-\A} 
  \mu (x^1, \vb)
  \int_{\p\O} \dd S_{x^2} \int_{n(x^2) \cdot v^2 > 0}\dd v^2 \{n(x^2) \cdot v^2\} \notag
\\& \ \ \ \ \times
  \int^{ T_0- \tb(x,v)}_{   
 T_0- \tb(x,v) - \min\big(\tb(x^2, v^2) ,
 \frac{T_0}{4}
\big)} \dd \tb^1  f(NT_0 -\tb(x,v) - \tb^1,x^2, v^2). \label{lower1}
\end{align} 
 
Now we focus on the integrand of \eqref{lower1}. Recall \eqref{def:T+}, we have  
\Be \notag
(NT_0 - \tb(x,v) - \tb^1)- (N-1)T_0= T_0 - \tb(x,v) - \tb^1
\in \Big[0, \min\Big(\tb(x^2, v^2) ,
 \frac{T_0}{4}
\Big)\Big]. 
\Ee
Now setting $y_* = X((N-1)T_0;t^2, x^2, v^2)$, $v_* = V((N-1)T_0;t^2, x^2, v^2)$ and $\alpha = T_0 - \tb(x,v) - \tb^1$, 
we have
\Be \label{lower2}
\eqref{lower1}
= \int^{\min\big(\tb(x^2, v^2) ,
 \frac{T_0}{4}
\big)}_{0}
 f
\big( (N-1)T_0, y_*, v_* \big) \dd \alpha.
\Ee
From \eqref{cond:tf}, we have $\tf(y_*, v_*) \in \big[0, \frac{T_0}{4}\big]$. Now applying \eqref{COV}, we conclude that 
 \Be \notag
\eqref{Doeblin_2} 
\geq \mathbf{1}_{ \tb(x,v) \leq \frac{T_0}{4}} (T_0)^{-4-\A}
  \mu (x^1, \vb)
  \iint_{\O \times \R^3} \mathbf{1}_{\tf(y,v)  \in [0, \frac{T_0}{4}]} 
  f((N-1)T_0, y,v) \dd v \dd y.  \notag
 \Ee
We conclude \eqref{est:Doeblin} by setting 
\Be \label{def:m}
\mathfrak{m} (x,v) :=  
\mathbf{1}_{ \tb(x,v) \leq \frac{T_0}{4}} (T_0)^{-4-\A} \mu (x^1, \vb).
\Ee 
\end{proof} 

An immediate consequence of Proposition \ref{prop:Doeblin 1}. follows.

\begin{proposition} \label{prop:Doeblin}

{\color{black}
Suppose $f$ solves \eqref{equation for F}, \eqref{diff_F} and satisfies \eqref{cons_mass_f}. Then for all $T_0\gg1$ and $N \in \mathbb{N}$, 
}
\Be \label{L1_coerc}
   \|f(NT_0)\|_{L^1_{x,v}}  \leq  (1-\| \mathfrak{m}\|_{L^1_{x,v}} )   \|f((N-1)T_0)\|_{L^1_{x,v}} 
   + 2 \| \mathfrak{m}\|_{L^1_{x,v}}  \| \mathbf{1}_{\tf\geq \frac{T_0}{4}} f((N-1)T_0)\|_{L^1_{x,v}}.
 \Ee 
Moreover, 
{\color{black}
there exists $T_0 = T_0 (\O)$, such that
}
\Be \label{est:m}
\| \mathfrak{m} \|_{L^1_{x,v}} := \mathfrak{m}_{T_0} 
\lesssim (T_0)^{-3 - \A} |\p\O| < 1. 
\Ee
\end{proposition}

\begin{proof}

We decompose 
\begin{align*}
 f((N-1)T_0,x,v) 
 = f_{ N-1 ,+}(x,v) - f_{ N-1 ,-}(x,v),\end{align*}
where 
\Be \notag
\begin{split}
& f_{N-1,+} (x,v ) = \mathbf{1}_{f((N-1)T_0,x,v)  \geq 0} f \big( (N-1)T_0,x,v \big),
\\& f_{N-1,-} (x,v ) = \mathbf{1}_{f((N-1)T_0,x,v)  < 0} |f\big( (N-1)T_0,x,v \big)|.
\end{split}
\Ee
Let $f_{\pm}(s,x,v)$ solve \eqref{equation for F} for $s \in [ (N-1)T_0,NT_0]$ with the initial data $f_{N-1,+}$ and $f_{N-1,-}$ at $s=(N-1)T_0$, respectively.
 
Now we apply Proposition \ref{prop:Doeblin 1} on $f_{\pm}(t,x,v)$ and conclude \eqref{est:Doeblin} for $f= f_+$ and $f=f_-$ respectively. We also note that 
\Be \notag 
\iint_{\O \times \R^3} f((N-1)T_0,x,v) \dd x \dd v=\iint_{\O \times \R^3} f_{N-1,+}(x,v) \dd x \dd v- \iint_{\O \times \R^3} f_{N-1,-}(x,v) \dd x \dd v =0.
\Ee
This implies,
\Be \label{pm equality}
\iint_{\O \times \R^3} f_{N-1, \pm}( x,v)  \dd x \dd v =\frac{1}{2} \iint_{\O \times \R^3} |f((N-1)T_0,x,v)| \dd x \dd v.
\Ee
From \eqref{est:Doeblin},
\Be \notag
f_{N-1, \pm}( x,v) \geq \mathfrak{m}(x,v) \iint f_{N-1,\pm}(x,v) \dd x \dd v -  
  \mathfrak{m}(x,v) \iint_{\O \times \R^3} \mathbf{1}_{\tf(x,v) \geq \frac{T_0}{4}} f_{N-1,\pm} (x,v ) \dd x \dd v
\Ee
Using \eqref{pm equality}, we have
\Be \label{lowerB:f}
f_{N-1, \pm}( x,v) \geq 
\underbrace{ 
\mathfrak{m}(x,v) \Big(
\frac{1}{2} \|f((N-1)T_0)\|_{L^1_{x,v}}  
  - \| \mathbf{1}_{\tf\geq \frac{T_0}{4}} f((N-1)T_0)\|_{L^1_{x,v}}
  \Big)
}_{\mathfrak{l}(x,v)}. 
\Ee 
Then we deduce
\Be \notag
\begin{split}
|f(NT_0,x,v)| 
& = |f_{N-1, +}( x,v) - f_{N-1, -}( x,v) + \mathfrak{l}(x,v) - \mathfrak{l}(x,v)|
\\& \leq  |f_{N-1, +}( x,v) - \mathfrak{l}(x,v)| +|f_{N-1, -}( x,v) - \mathfrak{l}(x,v)|.
\end{split}
\Ee 
From \eqref{lowerB:f},
\Be \label{upperNT:f}
|f(NT_0,x,v)| \leq f_{N-1, +}( x,v) + f_{N-1, -}( x,v) - 2 \mathfrak{l}(x,v).
\Ee
Note that $f_+(NT_0,x,v)+f_-(NT_0,x,v)$ solves \eqref{equation for F} with the initial datum $$f_{N-1,+} + f_{N-1,-}
= \big| f \big((N-1)T_0,x,v \big) \big|.$$ 
Using \eqref{pm equality}, \eqref{upperNT:f} and taking the integration on \eqref{lowerB:f} over $\O \times \R^3$, we derive
\Be \notag
\begin{split}
\|f(NT_0)\|_{L^1_{x,v}} 
& \leq \iint_{\O \times \R^3} f_{N-1, +}( x,v)  \dd x \dd v
+ \iint_{\O \times \R^3} f_{N-1, -}( x,v)  \dd x \dd v - \iint_{\O \times \R^3} 2 \mathfrak{l}(x,v)  \dd x \dd v
\\& = (1-\| \mathfrak{m}\|_{L^1_{x,v}} )   \|f((N-1)T_0)\|_{L^1_{x,v}} 
   + 2 \| \mathfrak{m}\|_{L^1_{x,v}}  \| \mathbf{1}_{\tf\geq \frac{T_0}{4}} f((N-1)T_0)\|_{L^1_{x,v}},
\end{split}
\Ee
therefore we prove \eqref{L1_coerc}.

To derive \eqref{est:m}, it suffices to bound $\| \mathbf{1}_{\tb(x,v) \leq \frac{T_0}{4}}  
\mu (x^1, \vb)\|_{L^1_{x,v}}$. 
From Lemma \ref{lem:COV}, Lemma \ref{conservative field}, and 
{\color{black}
$\tb (X(t-s,t,x,v), V(t-s,t,x,v)) = \tb(x,v)- s$,
}
we have
\Be \notag
\begin{split}
& \ \ \ \ \| \mathbf{1}_{\tb(x,v) \leq \frac{T_0}{4}} \mu (x^1, \vb) \|_{L^1_{x,v}} 
 = \int_{\gamma_{+}} \int_{ \max\{0,\tb(x,v)- \frac{T_0}{4}\}}  ^{\tb(x,v)} 
\mu (x^1, v) \{n(x) \cdot v\}\dd s \dd v \dd S_x \\
\\& \lesssim \int_{\gamma_{+}}  
\Big(  \mathbf{1}_{\tb(x,v) \leq \frac{T_0}{4}} \int_0^{\tb(x,v)} \dd s
+\mathbf{1}_{\tb(x,v) \geq \frac{T_0}{4}}\int^{\tb(x,v)}_{\tb(x,v)- \frac{T_0}{4}}  \dd s \Big) 
e^{-\frac{1}{2} |v|^2} \{n(x) \cdot v\} \dd v \dd S_x
\\& \leq \frac{T_0}{4} \int_{\p\O } \dd S_x \int_{n(x) \cdot v>0} 
e^{-\frac{1}{2} |v|^2} \{n(x) \cdot v\} \dd v \lesssim T_0 |\p\O|.  
\end{split}
\Ee
Combining the above bound with \eqref{def:m}, we conclude \eqref{est:m}.
\end{proof}

{\color{black}
\begin{remark}
Throughout this paper, we consider $\O = \mathbb{T}^2 \times \R_{+}$ and $|\p\O| = 1$. Thus, any $T_0 > 1$ satisfies \eqref{est:m}. In general, $T_0$ depends heavily on $|\p\O|$, otherwise $\mathfrak{m}_{T_0} > 1$ and it leads to a negative estimate for $L^1$ in \eqref{L1_coerc}.
\end{remark}
}

\subsection{Proof of weighted \texorpdfstring{$L^1$}{L1}-Estimates}

{\color{black}
In this section, we prove Theorem \ref{theorem_1}. 
We start with establishing
}
the uniform estimates of the following energies:
\Be \label{|||i}
\vertiii{f}_i :=
\|f \|_{L^1_{x,v}}
+ \frac{ 4 \mathfrak{m}_{T_0} }{ \varphi_{i-1} (\frac{3T_0}{4})} \| \varphi_{i-1}(\tf) f\|_{L^1_{x,v}}  
+ \frac{ 4e \mathfrak{m}_{T_0} }{ T_0 \varphi_{i-1} (\frac{3T_0}{4})} \| \varphi_{i}(\tf) f\|_{L^1_{x,v}},
\Ee
where $\| \mathfrak{m} \|_{L^1_{x,v}} := \mathfrak{m}_{T_0}$ (see \eqref{est:m}) and $\varphi_i$'s defined in \eqref{varphis} with $i = 2, 4$.

Here we first introduce the weight functions $\varphi_i$'s.

\begin{definition} \label{def:varphis}
For $0 < \delta < 1$, we set
\Be \label{varphis}
\begin{split} 
& \varphi_1 (\tau) := (e \ln (e+1))^{-1} (e+ \tau)\ln(e + \ln (e+ \tau)), 
\\&  \varphi_2 (\tau) := (e^2 \ln (e+1))^{-1}(e+ \tau)^2 \ln(e + \ln (e+ \tau)), 
\\& \varphi_3 (\tau) := e^{5 - \A} (\tau+e)^{\A - 5} \big(\ln (\tau+e)\big)^{-(1+\delta)}, 
\\& \varphi_4 (\tau) := e^{4 - \A} (\tau+e)^{\A - 4} \big(\ln (\tau+e)\big)^{-(1+\delta)}.
\end{split}
\Ee 
 
First, $\varphi_i$ satisfies \eqref{cond:varphi} for $i = 1,2,3,4$: for example, for $i = 4$,
\Be \notag
\begin{split}
\int_1^\infty \tau^{3 - \A} \varphi_4 (\tau)  \dd \tau
= \int_1^\infty \tau^{3 - \A} e^{4 - \A} (\tau+e)^{\A - 4} \big(\ln (\tau+e)\big)^{-(1+\delta)}  \dd \tau
& \lesssim  \int_1^\infty \frac{1}{(\tau+e) (\ln(\tau+e))^{1+\delta}}\dd s < \infty.
\end{split}
\Ee

Second, $\varphi_i$ satisfies
\Be\label{varphi|0}
\varphi_i (0)=1, \ \ \text{for} \ i= 1,2,3,4.
\Ee
 
Finally, we have
\Be \label{phi'}
\begin{split}
&  \varphi_2^\prime (\tau) 
\geq (e^2 \ln (e+1))^{-1} 2 (e+ \tau)\ln(e + \ln (e+ \tau))
\geq  2 e ^{-1} \varphi_1 (\tau),
\ \ \varphi_1^\prime (\tau) \geq 0,
\\& \varphi_4^\prime (\tau) 
= \big(\A - 4 - \frac{1+\delta}{\ln (\tau+e)}\big)e^{4 - \A} (\tau+e)^{\A - 5}  \big(\ln (\tau+e)\big)^{-(1+\delta)} 
\geq \varphi_3 (\tau), \ \
\varphi_3^\prime (\tau) \geq 0.
\end{split}   
\Ee 
\end{definition}

\begin{proposition} \label{prop:energy}
Choose $T_0 > 20$, such that for the constant $C$ in \eqref{energy_varphi},
\Be \label{cond:T0}
4 C (e + 3 T_0)  \Big( \varphi_i \big( \frac{3T_0}{4} \big) \Big)^{-1}
\leq \frac{1}{2}, \ \text{for} \ i = 1, 3.
\Ee
For any $N \in \mathbb{N}$, and $i = 2, 4$, 
\Be \label{energyi}
\begin{split}   
& \ \ \ \ \| f(NT_0) \|_{L^1_{x,v}}
    +    \frac{ 4\mathfrak{m}_{T_0} }{ \varphi_{i-1} (\frac{3T_0}{4})}
    \Big\{
2 \| \varphi_{i-1} (\tf) f  (NT_0)\|_{L^1_{x,v}} 
+  \frac{e}{T_0} \| \varphi_i(\tf) f(NT_0) \|_{L^1_{x,v}}
	\Big\}
\\& \leq \  \eqref{energyi}_* \times 
    \|f((N-1)T_0)\|_{L^1_{x,v}}  
\\& \quad\qquad + \frac{4 \mathfrak{m}_{T_0} }{ \varphi_{i-1} (\frac{3T_0}{4})}
 \Big\{ \frac{3}{4} \| \varphi_{i-1}(\tf)f((N-1)T_0)\|_{L^1_{x,v}}  
 +  \frac{e }{T_0  }
 \| \varphi_i(\tf) f((N-1)T_0) \|_{L^1_{x,v}}\Big\},   
\end{split}
\Ee
where $\eqref{energyi}_* := 1 - \mathfrak{m}_{T_0} \{
1 - \frac{ 4 C(e + 3 T_0) }{ \varphi_{i-1}(\frac{3T_0}{4})} \}$, with $\mathfrak{m}_{T_0}$ defined in \eqref{est:m}.
\end{proposition}

\begin{proof} 
As key steps, we apply Lemma \ref{lemma:energy} on $f(t,x,v)$ solving \eqref{equation for F} and \eqref{diff_F} with $\varphi_i$'s in \eqref{varphis}, and using \eqref{varphi|0}, we derive that, for $i = 2, 4$ and $(N-1)T_0 \leq t_*\leq NT_0$, 
\be \label{energy:log} 
\| \varphi_{i-1}(\tf) f(NT_0) \|_{L^1_{x,v}} +   \frac{3}{4} \int^{NT_0}_{t_*} |f|_{L^1_{\gamma_+}} \dd s
\leq  \|  \varphi_ {i-1}(\tf)f(t_*) \|_{L^1_{x,v}}  + C T_0 \| f(t_*) \|_{L^1_{x,v}}, 
\ee
and
\be \label{energy:phii}
\begin{split}
& \ \ \ \ \| \varphi_i(\tf) f(NT_0) \|_{L^1_{x,v}} + \int^{NT_0}_{(N-1)T_0}\{
 \| \varphi_i^\prime(\tf) f \|_{L^1_{x,v}}
 +  \frac{3}{4}  |f|_{L^1_{\gamma_+}}
\} \dd s
\\& \leq  \| \varphi_i(\tf) f((N-1)T_0) \|_{L^1_{x,v}}  + C T_0 \| f((N-1)T_0) \|_{L^1_{x,v}},
\end{split}
\ee
where we set $t_* = (N-1)T_0$ in \eqref{energy:phii}.

From \eqref{maximum}, \eqref{phi'} and \eqref{energy:log}, we derive that, for $i= 2, 4$, 
\Be \label{energy:phiii}
\begin{split}
\int^{NT_0}_{(N-1)T_0}
  \| \varphi_i^\prime(\tf) f \|_{L^1_{x,v}}  
& \geq \int^{NT_0}_{(N-1)T_0}
2 e^{-1} \| \varphi_{i-1}(\tf) f (t_*)\|_{L^1_{x,v}} \dd t_* 
\\& \geq 2 e^{-1} T_0 \| \varphi_{i-1}(\tf) f(NT_0) \|_{L^1_{x,v}} 
 - 2 e^{-1} C (T_0)^2 \| f((N-1)T_0) \|_{L^1_{x,v}}.
\end{split}
\Ee 
Applying \eqref{energy:phiii} on \eqref{energy:phii}, we conclude that, for $i = 2, 4$,  
\Be \label{energy:phi1}
\begin{split}
& \ \ \ \ \| \varphi_i(\tf) f(NT_0) \|_{L^1_{x,v}} + 
2 e^{-1} T_0 \| \varphi_{i-1}(\tf) f  (NT_0)\|_{L^1_{x,v}}
 + \frac{3}{4} \int^{NT_0}_{(N-1)T_0} |f|_{L^1_{\gamma_+}}
\\& \leq \| \varphi_i(\tf) f((N-1)T_0) \|_{L^1_{x,v}}  + C T_0 (1+ 2 e^{-1} T_0)\| f((N-1)T_0) \|_{L^1_{x,v}}.
\end{split}
\Ee
Note that from \eqref{phi'}, we have, for $i = 2, 4$,
\be \label{energy:phi2}
\mathbf{1}_{\tf \geq \frac{3T_0}{4}}
\leq \big( \varphi_{i-1}(\frac{3T_0}{4}) \big)^{-1} \varphi_{i-1}(\tf),
\ee 
Now we combine \eqref{L1_coerc} with \eqref{energy:log}-\eqref{energy:phi2} and $\mathfrak{m}_{T_0}$ in \eqref{est:m}
{\color{black}
with $|\p\O| = 1$, and obtain
}

\Be \label{L1_coerc'}
\begin{split}
  \|f(NT_0)\|_{L^1_{x,v}}  
 & \leq  (1- \mathfrak{m}_{T_0}) \|f((N-1)T_0)\|_{L^1_{x,v}} 
 + \frac{ 2 \mathfrak{m}_{T_0} }{ \varphi_{i-1}(\frac{3T_0}{4}) } \| \varphi_{i-1}(\tf)f((N-1)T_0)\|_{L^1_{x,v}}.
\end{split}
\Ee   
For $i = 2, 4$ and $T_0 \gg 1$ in \eqref{cond:T0}, considering
$\eqref{L1_coerc'} 
+ \frac{4 \mathfrak{m}_{T_0}}{\varphi_{i-1} (\frac{3T_0}{4})}
\big\{
\frac{1}{4} \eqref{energy:log} |_{t_* = (N-1) T_0} + \frac{e}{T_0} \eqref{energy:phi1}
\big\}$, then we conclude \eqref{energyi}.
\end{proof}

Now we are well equipped to prove Theorem \ref{theorem_1}.

\begin{proof}[\textbf{Proof of Theorem \ref{theorem_1}}]
Fix $T_0$ in \eqref{cond:T0} and recall norms of $\vertiii{\cdot}_2$ and $\vertiii{\cdot}_4$ in \eqref{|||i}. From \eqref{energyi}, for $i = 2, 4$, 
\Be\label{bound|||}
     \vertiii{f(NT_0)}_i \leq  \vertiii{f((N-1)T_0)}_i  \leq \cdots \leq \vertiii{f(0)}_i, \ \ \text{for all } N \in \mathbb{N}.   
\Ee 

\textbf{Step 1.} 
Under direct computation, we obtain
\be
\begin{split}
\frac{\dd}{\dd \tau} \Big( \frac{ \varphi_2 (\tau)}{\varphi_4 (\tau)} \Big)
\lesssim \frac{ (1+\delta) \ln(e + \ln (e+ \tau))
- (\A - 6) \ln(e + \ln (e+ \tau)) ( \ln (\tau+e) )^{\delta}}{( \ln (\tau+e) )^{-1} (\tau+e)^{\A - 5} },
\end{split}
\ee
which shows that the function $\varphi_2 (\tau ) / \varphi_4 (\tau)$ is decreasing when $\tau\gg 1$.
Thus, we can choose $M \gg 1$ satisfying \eqref{1-1/M} and \eqref{def:M}, such that
\be \label{varphi2 expression}
\begin{split}
\varphi_2 (\tf) 
& = \mathbf{1}_{\tf\geq M} \varphi_2 (\tf) + \mathbf{1}_{\tf< M} \varphi_2 (\tf) 
\\& \leq \mathbf{1}_{\tf\geq M}\frac{ \varphi_2 (M )}{\varphi_4 (M)}  \varphi_4 (\tf)
 +  \mathbf{1}_{\tf<M} M \varphi_1 (\tf),
\end{split}
\ee
where 
{\color{black}
we use $\varphi_2 (\tau) = \frac{e + \tau}{\tau} \varphi_1 (\tau)$ and $\frac{e + \tau}{\tau} < M$ for $M \gg 1$.
}

Applying \eqref{bound|||} for $i = 4$ and \eqref{varphi2 expression} with $M\gg1$, we obtain for $1 \leq N \in \mathbb{N}$, 
\Be \label{intp:varphi}
\begin{split}
  &  \ \ \ \ \frac{1}{M}  \| \varphi_2 (\tf) f((N-1)T_0) \|_{L^1_{x,v}} 
  \\& \leq \frac{1}{M}\frac{ \varphi_2 (M )}{\varphi_4 (M)} \| \varphi_4 (\tf) f((N-1)T_0) \|_{L^1_{x,v}}
   +  \| \varphi_1 (\tf) f((N-1)T_0) \|_{L^1_{x,v}}\\
  & \leq \frac{1}{M}\frac{ \varphi_2 (M )}{\varphi_4 (M)} 
  \frac{T_0  \varphi_{3} (\frac{3T_0}{4})}{ 4e \mathfrak{m}_{T_0} } \vertiii{ f(0) }_4
+  \| \varphi_1 (\tf) f((N-1)T_0) \|_{L^1_{x,v}}.
\end{split}
\Ee
After inputting \eqref{intp:varphi} into \eqref{energyi} for $i=2$, we derive that 
\Be \label{energy2}
\begin{split} 
\vertiii{ f(NT_0 )}_2 
\leq \eqref{energy2}_* \times \vertiii{ f((N-1)T_0 )}_2 
+ \frac{1}{M}\frac{ \varphi_2 (M )}{\varphi_4 (M)} \frac{ \varphi_3 (\frac{3T_0}{4}) }{ \varphi_1 (\frac{3T_0}{4})} \vertiii{ f(0) }_4,
\end{split}
\Ee
with
$\eqref{energy2}_* := \max \big\{ 
(1-  \mathfrak{m}_{T_0}  \{ 1 - \frac{ 4 C(e + 3T_0)  }{\varphi_1 (\frac{3T_0}{4})} \} ), \
( \frac{3}{4} + \frac{e}{T_0} ), \
(1- \frac{1}{M} )
\big\}$.

\medskip

\textbf{Step 2.} 
Using $T_0 > 20$ in \eqref{cond:T0}, we have $\frac{3}{4} + \frac{e}{T_0} < 1$. Thus, tentatively we make an assumption, which will be justified later behind \eqref{def:M}, 
\Be \label{1-1/M}
\Big( 1 + \frac{1}{M} \Big)^{-1}
\geq \max \Big\{ 
\big( 1-  \mathfrak{m}_{T_0}  \{ 1 - \frac{ 4 C(e + 3T_0)  }{\varphi_1 (\frac{3T_0}{4})} \} \big), \
\big( \frac{3}{4} + \frac{e}{T_0} \big), \
\big( 1- \frac{1}{M} \big)
\Big\}.
\Ee
For any $t\geq 0$, we choose $N_* \in \mathbb{N}$ such that $t \in [N_* T_0, (N_* +1)T_0]$. From \eqref{energy2} and \eqref{1-1/M}, we derive, for all $0 \leq N\leq N_*+1$,
\Be \label{est|||0}
\vertiii{f(NT_0)}_2
\leq \Big(1 +\frac{1}{M}\Big)^{-1}  \vertiii{f((N-1)T_0)}_2 +  \mathfrak{R}, 
\Ee 
where $\mathfrak{R} 
:= \frac{1}{M} \frac{ \varphi_2 (M )}{\varphi_4 (M)} \frac{ \varphi_3 (\frac{3T_0}{4}) }{ \varphi_1 (\frac{3T_0}{4})} \vertiii{f(0)}_4$.  
     
From \eqref{energy_varphi} and $0 \leq N_* T_0 \leq t$, there exists the constant $C > 0$, such that
\Be \label{first term in est:|||1}
\begin{split} 
\| \varphi(\tf) f(t) \|_{L^1_{x,v}}
\leq \| \varphi(\tf) f(N_* T_0) \|_{L^1_{x,v}} + C T_0 \| f(N_* T_0) \|_{L^1_{x,v}}.
\end{split}
\Ee 
Now applying \eqref{first term in est:|||1} first and using \eqref{est|||0} successively, we conclude that
\begin{align}
\vertiii{f(t)}_2 
& \lesssim_{T_0} \vertiii{f(N_* T_0)}_2
\leq \Big(1+ \frac{1}{M}\Big)^{-1}\vertiii{f((N_*-1)T_0)}_2
 + \mathfrak{R} \notag 
\\& \leq  \Big(1+ \frac{1}{M}\Big)^{-2}\vertiii{f((N_*-2)T_0)}_2
+  \Big(1+ \frac{1}{M}\Big)^{-1} \mathfrak{R}
+  \mathfrak{R} \notag
 \\& \leq \cdots  \leq  \Big(1 + \frac{1}{M}\Big)^{-N_*}\vertiii{f(0)}_2
+ (1 + M) \mathfrak{R}. \label{est:|||1}
\end{align} 
From $N_* T_0 \leq t \leq (N_* + 1) T_0$ and $1 \leq \frac{1+M}{M} \leq 2$, we get
\be \notag
\begin{split}
& \big(1 + \frac{1}{M}\big)^{-N_*} 
\lesssim \big( (1 + \frac{1}{M} )^{-M}\big)^{ \frac{N_* + 1}{M} } 
\lesssim e^{- \frac{N_* + 1}{2M}}
 \leq e^{- \frac{t}{2T_0M}},
\\& (1+M) \mathfrak{R} \leq 2 \frac{ \varphi_2 (M )}{\varphi_4 (M)} \frac{ \varphi_3 (\frac{3T_0}{4}) }{ \varphi_1 (\frac{3T_0}{4})} \vertiii{f(0)}_4.
\end{split}
\ee 
Then we have 
\Be \label{est:|||1'}
\|f \|_{L^1_{x,v}}
\leq \vertiii{f(t)}_2 \leq (\ref{est:|||1}) 
\lesssim \max  \big\{
e^{- \frac{t}{2T_0M}}, \ 
\varphi_2 (M ) / \varphi_4 (M) 
\big\} \times
\big\{
\vertiii{ f(0) }_2
+ \vertiii{ f(0) }_4
\big\}.
\Ee

\medskip

\textbf{Step 3.}
To make $|e^{- \frac{t}{2T_0M}} - { \varphi_2 (M )} / {\varphi_4 (M)}| \ll 1$ as $t \to \infty$, we set $M$ as follows:
\be \label{def:M}
M = t \big[ 2T_0 \ln ( 10 + t^{\A - 6}) \big]^{-1},
\ee   
so that
\be \label{decay rate:M}
\max \big\{
e^{- \frac{t}{2T_0M}}, \
 \varphi_2 (M ) / \varphi_4 (M) 
\big\} 
\lesssim_{T_0} (\ln\langle t\rangle )^{\A - 6 - \frac{\delta}{2}} \langle t\rangle^{6 - \A}.
\Ee
Clearly such a choice assures our precondition \eqref{1-1/M} for $t\gg1$.  

Now we claim that 
\be \label{est:initial data}
\vertiii{ f(0) }_2 + \vertiii{ f(0) }_4
\lesssim  \| e^{ \frac{1}{2} |v|^2+ \Phi (x) } f_0\|_{L^\infty_{x,v}}.
\ee
Note that it suffices to check that
$\| \varphi_4 (\tf) f_0 \|_{L^1_{x,v}}
\lesssim \| e^{ \frac{1}{2} |v|^2+ \Phi (x) } f_0 \|_{L^\infty_{x,v}}$.

Assume $\| e^{ \frac{1}{2} |v|^2 + \Phi (x) } f_0 \|_{L^\infty_{x,v}} < \infty$, from \eqref{COV}, \eqref{tb estimate} and $\A = [\frac{1}{\ln (a)}]$, then we derive
\be
\begin{split}
\iint_{\O \times \R^3} |\varphi_4 (\tf) f_0 (y, w) | \dd y \dd w
& \lesssim \int_{\gamma_{+}} \int_0^{t_{-}} 
(\tb (x, v))^{\A - 4}
e^{- \frac{1}{2} |v|^2 - \Phi(x) } |n(x) \cdot v|\dd s \dd v \dd S_x 
\\& = \int_{\gamma_{+}} \int_0^{t_{-}} 
(\tb (x, v))^{\A - 4} e^{- \frac{|v|^2}{2} } |n(x) \cdot v|\dd s \dd v \dd S_x
\\& \lesssim \int_{\gamma_{+}} \big( a^{ \frac{1}{2} v^2_3} \big)^{\A - 3} e^{- \frac{|v|^2}{2} } |v_3| \dd v \dd S_x 
\\& \lesssim \int_{v_3 < 0} a^{- |v_3|^2} \dd v_3 < \infty,
\end{split}
\ee
and this concludes the claim. 
Finally, together with \eqref{est:|||1'}, \eqref{decay rate:M} and \eqref{est:initial data}, we prove \eqref{est:theorem_1}.
\end{proof}

\section{Estimates on Exponential Moments}
\label{sec: exponential moments}

Now we are able to show the asymptotic behavior of the exponential moments. The main purpose of this section to prove Theorem \ref{theorem}.


\subsection{Some preparation on Exponential Moments}


{\color{black}
To estimate the exponential moments, we include two weight functions: (i) a time dependent weight function $\varrho (t)$, and (ii) a time independent weight function $w^\prime (x, v)$, which is constant along the characteristic trajectory \eqref{characteristics}. Then we 
}
consider the stochastic cycle representation of $\varrho (t) w^\prime (x, v) f(t, x, v)$.

\begin{lemma} \label{sto_cycle_2}

{\color{black}
Suppose $f (t, x, v)$ solves \eqref{equation for F} and \eqref{diff_F} with $0 = t_* \leq t$.
Consider a time dependent function $\varrho (t)$ and a time independent function $w^\prime (x, v)$, which is constant along the characteristic \eqref{characteristics}. Then for $k \geq 1$,
}
\begin{align}
   & \varrho (t) w^\prime (x, v) f(t, x, v) 
   = \mathbf{1}_{t^1 < 0} \
   \varrho (t) w^\prime (X(0; t, x, v), V(0; t, x, v)) f(0, X, V) 
\label{expand_k1}
\\& + w^\prime \mu (x^1, \vb) \sum\limits^{k-1}_{i=1}  \int_{\prod_{j=1}^{i} \mathcal{V}_j} \Big\{ \mathbf{1}_{t^{i+1} < 0 \leq t^{i}} 
     \varrho(0) w^\prime (X(0; t^i, x^i, v^i), V(0; t^i, x^i, v^i)) f (0, X, V) \Big\}
      \dd \tilde{\Sigma}_{i}
\label{expand_k2}
    \\& + w^\prime \mu (x^1, \vb) \sum\limits^{k-1}_{i=1}  \int_{\prod_{j=1}^{i} \mathcal{V}_j} \mathbf{1}_{0 \leq t^{i}}
     \Big\{ \int^{t^{i}}_{ \max(0, t^{i+1})}  
     \varrho^\prime (s) w^\prime (X(s; t^i, x^i, v^i), V(s; t^i, x^i, v^i))
\notag     
\\& \hspace{8cm} \times f(s, X(s; t^i, x^i, v^i), V(s; t^i, x^i, v^i)) \dd s 
      \Big\} \dd \tilde{\Sigma}_{i}
\label{expand_k3}
    \\& + w^\prime \mu (x^1, \vb) \int_{\prod_{j=1}^{k } \mathcal{V}_j}   
    \mathbf{1}_{t^{k} \geq 0} \
    \varrho (t^k) w^\prime f (t^{k}, x^{k}, v^{k})
     \dd \tilde{\Sigma}_{k}, 
\label{expand_k4}
\end{align} 
where 
$\dd \tilde{\Sigma}_{i} := \frac{ \dd \sigma_{i}}{\mu (x^{i+1}, v^{i}) w^\prime (x^i, v^{i})} \dd \sigma_{i-1} \cdots \dd \sigma_1$, with $\dd \sigma_j = \mu (x^{j+1}, v^{j}) \{ n(x^j) \cdot v^j \} \dd v^j$.
Here, $(X,V)$ solves \eqref{characteristics}.
\end{lemma}

\begin{proof}
Following Lemma \ref{sto_cycle_1} and Remark \ref{probability}, we obtain this Lemma.
\end{proof}

\medskip

We start with a simple case when $w^\prime (x, v) \equiv 1$.
Applying Lemma \ref{sto_cycle_2}, we derive the 
stochastic cycle representation of $\varrho (t) f(t, x, v)$ as follows.
\begin{align}
   \varrho (t) f(t, x, v) 
   = & \mathbf{1}_{t^1 < 0} \
   \varrho (t) f(0, X(0; t, x, v), V(0; t, x, v)) 
\label{expand_G1}
\\& + \mu (x^1, \vb) \sum\limits^{k-1}_{i=1} \int_{\prod_{j=1}^{i} \mathcal{V}_j} \Big\{ \mathbf{1}_{t^{i+1} < 0 \leq t^{i}} 
     \varrho(0) f (0, X(0; t^i, x^i, v^i), V(0; t^i, x^i, v^i)) \Big\}
      \dd \tilde{\Sigma}_{i}
\label{expand_G2}
    \\& + \mu (x^1, \vb) \sum\limits^{k-1}_{i=1}   \int_{\prod_{j=1}^{i} \mathcal{V}_j} \mathbf{1}_{0 \leq t^{i}}
     \Big\{ \int^{t^{i}}_{ \max(0, t^{i+1})}  
     \varrho^\prime (s) f(s, X(s; t^i, x^i, v^i), V(s; t^i, x^i, v^i)) \dd s 
      \Big\} \dd \tilde{\Sigma}_{i}
\label{expand_G3}
    \\& + \mu (x^1, \vb) \int_{\prod_{j=1}^{k } \mathcal{V}_j}   
    \mathbf{1}_{t^{k} \geq 0} \
    \varrho (t^k) f (t^{k}, x^{k}, v^{k})
     \dd \tilde{\Sigma}_{k}, 
\label{expand_G4}
\end{align} 
where 
$\dd \tilde{\Sigma}_{i} = \frac{ \dd \sigma_{i}}{\mu (x^{i+1}, v^{i})} \dd \sigma_{i-1} \cdots \dd \sigma_1$, with $\dd \sigma_j = \mu (x^{j+1}, v^{j}) \{ n(x^j) \cdot v^j \} \dd v^j$, and $(X,V)$ solves \eqref{characteristics}.

Here, we put emphasis on \eqref{expand_G3} and \eqref{expand_G4} since \eqref{expand_G1} and \eqref{expand_G2} can be controlled by Theorem \ref{theorem_1} and initial condition, which will be shown in the proof of Theorem \ref{theorem}.

To estimate \eqref{expand_G3}, we apply Proposition \ref{prop:mapV} on $\mathcal{V}_{j}$ for $j = i-2, i-1$ with \eqref{tb estimate}, \eqref{small tb estimate}, and derive
\begin{equation} \label{forcing_2}
\begin{split}
& \ \ \ \ \int_{\mathcal{V}_{i-2}} \dd \sigma_{i-2} \int_{\mathcal{V}_{i-1}} \dd \sigma_{i-1}
\int_{\mathcal{V}_i} 
\mathbf{1}_{t^{i+1} < 0 \leq t^i}
\int^{t^{i}}_{0} \varrho^\prime(s) f(s, X(s; t^i, x^i, v^i), V(s; t^i, x^i, v^i)) 
\{n(x^i) \cdot v^i\}  \dd s \dd v^i   
\\& \lesssim \int_0^{t^{i-2}} \dd \tb^{i-1}  \int_{\p\O}  
\frac{\dd S_{x^{i-1}}}{\tb^{i-2}} 
\sum\limits_{m, n \in \mathbb{Z}} \mu_{\Theta} (x^{i-1}, v^{m, n}_{i-2, \mathbf{b}})
\int_0^{t^{i-2}- \tb^{i-1}} \dd \tb^{i-2}  \int_{\p\O}   
\frac{\dd S_{x^i}}{\tb^{i-1}} \sum\limits_{m, n \in \mathbb{Z}} \mu_{\Theta} (x^{i}, v^{m, n}_{i-1, \mathbf{b}}) 
\\& \ \ \ \ \times \underbrace{ \int_{\mathcal{V}_i} 
\mathbf{1}_{t^{i+1} < 0 \leq t^i}
\int^{t^{i}}_{0} \varrho^\prime(s) | f(s, X(s; t^i, x^i, v^i), V(s; t^i, x^i, v^i)) | \dd s
\{n(x^i) \cdot v^i\} \dd v^i}_{\eqref{forcing_2}^*}, 
\end{split}
\end{equation}
with $t^{i-1} = t^{i-2} - \tb^{i-2}$, \ 
$t^{i}= t^{i-1} - \tb^{i-1}$ and $v^{m, n}_{i-1, \mathbf{b}} = \vb (x^{i}, v^{m, n}_{i-1})$, \ 
$v^{m, n}_{i-2, \mathbf{b}} = \vb (x^{i-1}, v^{m, n}_{i-2})$.

\medskip

Now we can control \eqref{expand_G3} via the following lemma:

\begin{lemma} \label{lem:bound1}
Suppose $f(t,x,v)$ solves \eqref{equation for F}, \eqref{diff_F} and $(X,V)$ solves \eqref{characteristics}, for $0 \leq t^i \leq t$ and $i = 3, \cdots , k-1$, 
\Be \label{bound1:expand_h}
\int_{\prod_{j=1}^{i} \mathcal{V}_j}      
\mathbf{1}_{t^{i+1} < 0 \leq t^{i }}
\int^{t^{i}}_{0} \varrho^\prime(s) f(s, X(s; t^i, x^i, v^i), V(s; t^i, x^i, v^i)) \dd s 
\dd \tilde{\Sigma}_{i}
\lesssim  \int^t_0 \| \varrho^\prime (s) f(s) \|_{L^1_{x,v}} \dd s.
\Ee
where 
$\dd \tilde{\Sigma}_{i} = \frac{ \dd \sigma_{i}}{\mu (x^{i+1}, v^{i})} \dd \sigma_{i-1} \cdots \dd \sigma_1$, with $\dd \sigma_j = \mu (x^{j+1}, v^{j}) \{ n(x^j) \cdot v^j \} \dd v^j$, and $(X,V)$ solves \eqref{characteristics}.
\end{lemma}

\begin{proof}

\textbf{Step 1.} 
For \eqref{bound1:expand_h}, it suffices to prove this upper bound for $i = 2,...,k-1$, 
\Be \label{forcing}
\int_{\prod_{j=1}^{i-1} \mathcal{V}_j}  
\int_{\mathcal{V}_i} 
\mathbf{1}_{t^{i+1} < 0 \leq t^i}
\int^{t^{i}}_{0} \varrho^\prime(s) | f(s, X(s; t^i, x^i, v^i), V(s; t^i, x^i, v^i)) | \dd s
\{n(x^i) \cdot v^i\} \dd v^i
\dd \sigma_{i-1}  \cdots \dd \sigma_1.
\Ee  
Applying Propsition \ref{prop:mapV} as in \eqref{forcing_2}, we bound the above integration as  
\Be \label{est2:forcing}
\begin{split}
\eqref{forcing} \lesssim 
&\int_{\mathcal{V}_1} \dd \sigma_1 \cdots \int_{\mathcal{V}_{i-3}} \dd \sigma_{i-3}
\int^{t^{i-2}}_0 \dd \tb^{i-1}
 \int_0^{t^{i-2}- \tb^{i-1}} \dd \tb^{i-2} \int_{\p\O} \dd S_{x^i}
\\& \times
\underbrace{\bigg(\int_{\p\O} 
\frac{\sum\limits_{m, n \in \mathbb{Z}} \mu (x^{i-1}, v^{m, n}_{i-2, \mathbf{b}})}{|\tb^{i-2}|} 
\times
\frac{\sum\limits_{m, n \in \mathbb{Z}} \mu (x^{i}, v^{m, n}_{i-1, \mathbf{b}})}{|\tb^{i-1}|}   \dd S_{x^{i-1}} \bigg)}_{\eqref{est2:forcing}_*}
\times \eqref{forcing_2}^*.
\end{split}
\Ee

\medskip

\textbf{Step 2.} We claim that 
\Be \label{est3:forcing}
\eqref{est2:forcing}_* \lesssim \mathbf{1}_{\tb^{i-1} \leq \tb^{i-2}}
\langle \tb^{i-2}\rangle^{4 - \A}
+ \mathbf{1}_{\tb^{i-1} \geq \tb^{i-2}}
\langle \tb^{i-1}\rangle^{4 - \A}.
\Ee
In order to prove this claim, we split into the following two cases: 

\textit{Case 1:} $\tb^{i-1} \leq \tb^{i-2}$.
Using \eqref{t>1,i-2} and \eqref{t<1,i-2} in Lemma \ref{lem: sum of mu}, we bound 
\Be \label{bound1_1and5}
\begin{split}
\frac{\sum\limits_{m, n \in \mathbb{Z}} \mu (x^{i-1}, v^{m, n}_{i-2, \mathbf{b}})}{|\tb^{i-2}|}  
\lesssim  
\mathbf{1}_{\tb^{i-2} \leq 1} \frac{1}{
|\tb^{i-2}| 
}+ \mathbf{1}_{\tb^{i-2} \geq 1}
\frac{1}{|\tb^{i-2}|^{\A - 4} }, 
\end{split}
\Ee
Replacing $i$ with $i+1$ in \eqref{t>1,i-2} and \eqref{t<1,i-2}, we bound
\Be \label{bound2_1and5}
\begin{split}
\frac{\sum\limits_{m, n \in \mathbb{Z}} \mu (x^{i}, v^{m, n}_{i-1, \mathbf{b}})}{|\tb^{i-1}|}  
& \lesssim \underbrace{\sum\limits_{|a| < 2, |b| < 2} \mathbf{1}_{\tb^{i-1} \leq 1}  \frac{1}{|\tb^{i-1}|} \mu \big(x^{i}, \frac{|x^i + (a, b) - x^{i-1}|}{|\tb^{i-1}|} \big)}_{\eqref{bound2_1and5}_1} 
\\& + \underbrace{\mathbf{1}_{\tb^{i-1} \leq 1} \frac{1}{|\tb^{i-1}|} e^{-\frac{1}{2 (\tb^{i-1})^2}}}_{\eqref{bound2_1and5}_2} 
+ \underbrace{
\mathbf{1}_{\tb^{i-1} \geq 1} 
|\tb^{i-1}|^{4 - \A}}_{\eqref{bound2_1and5}_3}.
\end{split}
\Ee

For $\eqref{bound2_1and5}_1$, we employ a change of variables, for $x^{i} \in \p\O$, $|a| < 2$, $|b| < 2$ and $0 < \tb^{i-1} \leq 1$, 
\be
x^{i-1} \in  \p\O
\mapsto z := \frac{1}{\tb^{i-1}} (x^{i-1} + (a, b) - x^{i}) \in \mathfrak{S}^{a, b}_{x^{i}, \tb^{i-1}},
\ee 
where the image $\mathfrak{S}^{a, b}_{x^{i}, \tb^{i-1}}$ of the map is a two dimensional smooth plane. 
Using the local chart of $\p\O$, we have 
$\dd S_{x^{i-1}} \lesssim |\tb^{i-1}|^2 \dd S_z$.
From this change of variables and \eqref{bound1_1and5}, we conclude that
\Be \label{bound3_1and5_1}
\begin{split}
& \ \ \ \ \mathbf{1}_{\tb^{i-1} \leq \tb^{i-2}}  \int_{\p\O} \frac{\sum\limits_{m, n \in \mathbb{Z}} \mu (x^{i-1}, v^{m, n}_{i-2, \mathbf{b}})}{|\tb^{i-2}|} 
\times \eqref{bound2_1and5}_1 \ \dd S_{x^{i-1}}
\\& \lesssim \mathbf{1}_{\tb^{i-1} \leq \tb^{i-2}} \eqref{bound1_1and5} \times
\sum\limits_{|a| < 2, |b| < 2}
\int_{ \mathfrak{S}^{a, b}_{x^{i}, \tb^{i-1}} } \mathbf{1}_{\tb^{i-1} \leq 1} e^{- \frac{1}{2} |z|^2 } |\tb^{i-1}| \dd S_{z} 
\\& \lesssim  \mathbf{1}_{\tb^{i-1} \leq \tb^{i-2}}
 \Big\{
 \mathbf{1}_{\tb^{i-2} \leq 1} \frac{1}{
|\tb^{i-2}|} + \mathbf{1}_{\tb^{i-2} \geq 1}
\frac{1}{|\tb^{i-2}|^{\A - 4}} \Big\}
\mathbf{1}_{\tb^{i-1} \leq 1}|\tb^{i-1}|
\\& \lesssim \mathbf{1}_{\tb^{i-1} \leq \tb^{i-2}}\Big\{
 \mathbf{1}_{\tb^{i-2} \leq 1} \frac{|\tb^{i-1}|}{|\tb^{i-2}|} + \mathbf{1}_{\tb^{i-2} \geq 1}
\frac{1}{|\tb^{i-2}|^{\A - 4}}
 \Big\} 
\lesssim \mathbf{1}_{\tb^{i-2} \leq 1} + \mathbf{1}_{\tb^{i-2} \geq 1}
\frac{1}{|\tb^{i-2}|^{\A - 4}}.
\end{split}
\Ee
 
For $\eqref{bound2_1and5}_2$, since $e^{-\frac{1}{2 t^2}} \lesssim t^2$ for $0 < t \leq 1$, then we have
\Be \label{bound3_1and5_2}
\begin{split}
& \ \ \ \ \mathbf{1}_{\tb^{i-1} \leq \tb^{i-2}} \int_{\p\O} 
\frac{\sum\limits_{m, n \in \mathbb{Z}} \mu (x^{i-1}, v^{m, n}_{i-2, \mathbf{b}})}{|\tb^{i-2}|} 
\times \eqref{bound2_1and5}_2 \ \dd S_{x^{i-1}}
\\& \lesssim \mathbf{1}_{\tb^{i-1} \leq \tb^{i-2}} \eqref{bound1_1and5} \times
\mathbf{1}_{\tb^{i-1} \leq 1} \frac{1}{|\tb^{i-1}|} e^{-\frac{1}{2 (\tb^{i-1})^2}}
\int_{\p\O}  
 \dd S_{x^{i-1}} 
\\& \lesssim \mathbf{1}_{\tb^{i-1} \leq \tb^{i-2}}
 \Big\{
 \mathbf{1}_{\tb^{i-2} \leq 1} \frac{1}{
|\tb^{i-2}| 
} + \mathbf{1}_{\tb^{i-2} \geq 1}
\frac{1}{|\tb^{i-2}|^{\A - 4}} \Big\}
\mathbf{1}_{\tb^{i-1} \leq 1} |\tb^{i-1}|
\\& \leq \mathbf{1}_{\tb^{i-1} \leq \tb^{i-2}} \frac{|\tb^{i-1}|}{|\tb^{i-2}|}
\mathbf{1}_{\tb^{i-2} \leq 1} 
+ \mathbf{1}_{\tb^{i-2} \geq 1}
\frac{1}{|\tb^{i-2}|^{\A - 4}}
\\& \leq \mathbf{1}_{\tb^{i-2} \leq 1} + \mathbf{1}_{\tb^{i-2} \geq 1}
\frac{1}{|\tb^{i-2}|^{\A - 4}}. 
\end{split}
\Ee

For $\eqref{bound2_1and5}_3$, from
$\mathbf{1}_{\tb^{i-1} \geq 1} 
|\tb^{i-1}|^{4 - \A} \lesssim \mathbf{1}_{\tb^{i-1} \geq 1}$, we derive 
\Be \label{bound3_1and5_3}
\begin{split}
& \ \ \ \ \mathbf{1}_{\tb^{i-1} \leq \tb^{i-2}} \int_{\p\O} 
\frac{\sum\limits_{m, n \in \mathbb{Z}} \mu (x^{i-1}, v^{m, n}_{i-2, \mathbf{b}})}{|\tb^{i-2}|} 
\times \eqref{bound2_1and5}_3 \ \dd S_{x^{i-1}}
\\& \lesssim \mathbf{1}_{\tb^{i-1} \leq \tb^{i-2}} \eqref{bound1_1and5} \times
\mathbf{1}_{\tb^{i-1} > 1} |\tb^{i-1}|^{4 - \A} \int_{\p\O} \dd S_{x^{i-1}} 
\\& \lesssim  \mathbf{1}_{\tb^{i-1} \leq \tb^{i-2}}
 \Big\{
 \mathbf{1}_{\tb^{i-2} \leq 1} \frac{1}{
|\tb^{i-2}| 
}+ \mathbf{1}_{\tb^{i-2} \geq 1}
\frac{1}{|\tb^{i-2}|^{\A - 4}} \Big\} \mathbf{1}_{\tb^{i-1} \geq 1}
\\& \lesssim \mathbf{1}_{\tb^{i-2} \geq 1} \frac{1}{|\tb^{i-2}|^{\A - 4}}.
\end{split}
\Ee

Collecting estimate from \eqref{bound3_1and5_1}-\eqref{bound3_1and5_3}, we deduce that
\Be \label{bound3_1and5}
\mathbf{1}_{\tb^{i-1} \leq \tb^{i-2}} \eqref{est2:forcing}_*  \lesssim 
 \mathbf{1}_{\tb^{i-2} \leq 1} 
+ \mathbf{1}_{\tb^{i-2} \geq 1}
 {|\tb^{i-2}|^{4 - \A}}.
\Ee

\textit{Case 2:} $\tb^{i-1} \geq \tb^{i-2}$.
We change the role of $i-1$ and $i-2$ and follow the argument of the previous case.   We employ a change of variables, for $x^{i-1} \in \p\O$ $|a| < 2$, $|b| < 2$  and $0 < \tb^{i-2} \leq 1$, 
\be
x^{i-2} \in  \p\O
\mapsto z := \frac{1}{\tb^{i-2}} (x^{i-2}-x^{i-1}) \in \mathfrak{S}_{x^{i-1}, \tb^{i-2}}, \ee
with $\dd S^{a, b}_{x^{i-2}} \lesssim |\tb^{i-2}|^2 \dd S_z$. 
Then we can conclude that 
\Be \label{bound4_1and5}
\begin{split}
\mathbf{1}_{\tb^{i-1} \geq \tb^{i-2}} \eqref{est2:forcing}_*  \lesssim 
 \mathbf{1}_{\tb^{i-1} \leq 1} 
+ \mathbf{1}_{\tb^{i-1} \geq 1}
 {|\tb^{i-1}|^{4 - \A}}.
\end{split}
\Ee
Therefore, we show \eqref{est3:forcing}.
 
\medskip 
 
\textbf{Step 3.} Now we apply \eqref{est3:forcing} on \eqref{est2:forcing}. Then we have 
\begin{align}
\eqref{forcing}
& \lesssim \int_{\mathcal{V}_1} \dd \sigma_1 \cdots \int_{\mathcal{V}_{i-3}} \dd \sigma_{i-3}  
\int^{t^{i-2}}_0 
\frac{\dd \tb^{i-1}}{\langle \tb^{i-1} \rangle^{\A - 4}}
\int_0^{\min\{t^{i-2}- \tb^{i-1},  \tb^{i-1}  \}} \dd \tb^{i-2} \int_{\p\O} \dd S_{x^i} \times \eqref{forcing_2}^* \label{est_a:forcing}
\\& \ \ \ \ + \int_{\mathcal{V}_1} \dd \sigma_1 \cdots \int_{\mathcal{V}_{i-3}} \dd \sigma_{i-3} 
\int^{t^{i-2}}_0 \frac{\dd \tb^{i-2}}{\langle  \tb^{i-2}\rangle^{\A - 4}}
\int_0^{ \min\{t^{i-2}- \tb^{i-2},\tb^{i-2}  \}} \dd \tb^{i-1} \int_{\p\O} \dd S_{x^i} \times \eqref{forcing_2}^*.\label{est_b:forcing}
\end{align} 

For \eqref{est_a:forcing},  we employ the change of variables 
\be \notag
(x^{i}, \tb^{i-2}, v^{i}) 
\mapsto (y, w) = (X(s; t^{i-2} -\tb^{i-2} - \tb^{i-1}, x^{i}, v^{i}), V(s; t^{i-2} -\tb^{i-2} - \tb^{i-1}, x^{i}, v^{i})) \in \O \times\R^3, 
\ee
and we have
$|n(x^i) \cdot v^i| \dd S_{x^i} \dd \tb^{i-2} \dd v^i \lesssim \dd y \dd w$ from \eqref{COV}. 
From $0 \leq t^i \leq t$ and $\A \geq 8$, we bound \eqref{est_a:forcing} as
\Be \notag
\begin{split}
\eqref{est_a:forcing}
& \leq \int_{\mathcal{V}_1} \dd \sigma_1 \cdots \int_{\mathcal{V}_{i-3}} \dd \sigma_{i-3}
\int^{t^{i-2}}_0 \dd \tb^{i-1}\langle \tb^{i-1} \rangle^{4 - \A} 
\int^{t^{i}}_{0} \varrho^\prime(s)
\iint_{\O \times\R^3 } 
|f (s, y, w)|  \dd y \dd w \dd s
\\& \lesssim \int^t_0 \| \varrho^\prime (s) f(s) \|_{L^1_{x,v}} \dd s. 
\end{split}
\Ee

A bound of \eqref{est_b:forcing} can be derived similarly, by using the change of variables 
\be \notag
(x^{i}, \tb^{i-1}, v^{i}) 
\mapsto (y, w) = (X(s; t^{i-2} -\tb^{i-2} - \tb^{i-1}, x^{i}, v^{i}), V(s; t_{i-2} -\tb^{i-2} - \tb^{i-1}, x^{i}, v^{i})) \in \O \times\R^3, 
\ee
with $
|n(x^i) \cdot v^i| \dd S_{x^i} \dd \tb^{i-1} \dd v^i \lesssim \dd y \dd w$.
\end{proof}

Next, we control \eqref{expand_G4} by establishing the following estimate:

\begin{lemma} \label{lem:small_largek}
Consider $(X,V)$ solving \eqref{characteristics}, there exists $\mathfrak{C}= \mathfrak{C}(\O)>0$ (see \eqref{choice:k} for the precise choice),
such that  
\Be \label{small_largek}
\text{if }  \  k \geq \mathfrak{C}t,  \text{ then } 
\sup_{(x,v) \in \bar{\O} \times \R^3}  \Big(\int_{\prod_{j=1}^{k -1} \mathcal{V}_j}   
    \mathbf{1}_{t^{k} (t,x,v,v^1,\cdots, v^{k-1}) \geq 0 } \ \dd \sigma_1 \cdots \dd \sigma_{k-1}\Big) \lesssim e^{-t},
\Ee  
where $\dd \sigma_j = \mu (x^{j+1}, v^{j}) \{ n(x^j) \cdot v^j \} \dd v^j$ in \eqref{def:sigma measure}.
\end{lemma}

\begin{proof}
 
From \eqref{estimate on delta}, we have 
\Be \notag
\int_{n(x) \cdot v > 0}
\mathbf{1} _{\delta> \tb (x, v)} \mu (\xb, v) |n(x) \cdot v|\dd v
\lesssim C \delta^2.
\Ee
Thus we define $\mathcal{V}^{\delta}_i := \{ v^{i} \in \mathcal{V}_i:  {| n(x^{i}) \cdot v^{i} |} < \delta \}$, and derive that
\Be \notag
\int_{\mathcal{V}^{\delta}_j} \dd \sigma_j \leq C \delta^2.
\Ee

On the other hand, since $t_{\mathbf{b}} (x^{i}, v^{i}) \gtrsim {| n(x^{i}) \cdot v^{i} |}$, we derive that for  $v^{i} \in \mathcal{V}_i \backslash \mathcal{V}^{\delta}_i$,
\Be \notag
t_{\mathbf{b}} (x^{i}, v^{i}) \geq C_{\Omega} \delta.
\Ee
If $t_{k}(t,x,v^1,\cdots, v^{k-1}) \geq 0$, we conclude such $v^{i} \in \mathcal{V}_i \backslash \mathcal{V}^{\delta}_i$ can exist at most $[\frac{t  }{C_{\Omega} \delta}] + 1$ times. Denote the combination 
$\begin{pmatrix}
M \\ N
\end{pmatrix}= \frac{M(M-1) \cdots (M-N+1)}{N(N-1) \cdots 1}= \frac{M!}{N! (M-N)!}$ for $M,N \in \mathbb{N}$ and $M\geq N$. 
From $0<\delta \ll 1$, we have
\Be \label{sum_bound}
\begin{split} 
& \ \ \ \ \int_{\prod_{j=1}^{k-1} \mathcal{V}_j}   \mathbf{1}_{t_{k} (t,x,v^1,\cdots, v^{k-1}) \geq 0} \ \dd \sigma_{k-1} \cdots \dd \sigma_1
\\& \leq \sum\limits^{[\frac{t}{C_{\Omega} \delta}] + 1}_{m=0} 
\begin{pmatrix}
k
 \\
m
\end{pmatrix} 
\big( \int_{\mathcal{V}_i^\delta} \dd \sigma_i \big)^{k-m}
\leq (C\delta^2)^{k - [\frac{t  }{C_{\Omega} \delta}]}
  \underbrace{ \sum\limits^{[\frac{t  }{C_{\Omega} \delta}] + 1}_{m=0} 
\begin{pmatrix}
k
 \\
m
\end{pmatrix}}_{\eqref{sum_bound}_*}.
\end{split}
\Ee

Recall the Stirling's formula,
\Be \label{Stirling}
\sqrt{2 \pi} k^{k+\frac{1}{2}} e^{-k} \leq k ! \leq k^{k+\frac{1}{2}} e^{-k+1}.
\Ee
Using $(1 + \frac{1}{\mathfrak{a}-1})^{\mathfrak{a}-1} \leq e$ and \eqref{Stirling}, we have for $\mathfrak{a} \geq 2$,
\Be \notag
\begin{split}
\begin{pmatrix}
k \\
\frac{k}{\mathfrak{a}}
\end{pmatrix} 
  = \frac{k !}{ (k - \frac{k}{\mathfrak{a}}) ! \frac{k}{ \mathfrak{a}} !} 
& \leq 
 (\frac{\mathfrak{a}}{\mathfrak{a}-1})^{\frac{\mathfrak{a}}{\mathfrak{a}-1} k} \mathfrak{a}^{\frac{k}{\mathfrak{a}}} \sqrt{\frac{\mathfrak{a}^2}{k (\mathfrak{a}-1)}} 
\\&  = 
 \frac{1}{\sqrt{k}} \bigg( \mathfrak{a}^{\frac{1}{\mathfrak{a}}} \big( \frac{\mathfrak{a}}{\mathfrak{a}-1} \big)^{\frac{\mathfrak{a}}{\mathfrak{a}-1}} \bigg)^k \sqrt{\frac{ \mathfrak{a}^2}{\mathfrak{a}-1}}
  \leq
  \frac{1}{\sqrt{k}} (e \mathfrak{a})^{\frac{k}{\mathfrak{a}}} \sqrt{\frac{\mathfrak{a}^2}{ \mathfrak{a}-1}}.
\end{split}
\Ee
Hence, we derive that
\Be \label{est:sum_com}
\begin{split}
\sum\limits^{[\frac{k}{\mathfrak{a}}]}_{i=1} 
\begin{pmatrix}
k \\
i
\end{pmatrix} \leq 
\frac{k}{\mathfrak{a}} 
\begin{pmatrix}
k \\
\frac{k}{\mathfrak{a}}
\end{pmatrix} \leq \frac{e}{2 \pi} \sqrt{\frac{k}{\mathfrak{a}}} (e \mathfrak{a})^{\frac{k}{\mathfrak{a}}}.
\end{split}
\Ee

Now we estimate $\eqref{sum_bound}_*$. 
For fixed $0< \delta \ll 1$ which is independent of $t$, we choose  
\Be \label{choice:k}
\mathfrak{a} \in \mathbb{N} \ \text{ such that }
(\delta^{2 \mathfrak{a}} e \mathfrak{a})^{\frac{1}{C_\O \delta}} \leq e^{-2}, \ \text{ and set } 
k :=  \mathfrak{a} \Big(\big[\frac{t  }{C_{\Omega} \delta}\big] + 1\Big).
\Ee 
Using \eqref{est:sum_com}, we have
\Be \notag
\eqref{sum_bound}_* \lesssim   \sqrt{\big[\frac{t  }{C_{\Omega} \delta}\big] + 1} \Big( e \frac{k}{[\frac{t  }{C_{\Omega} \delta}] + 1} \Big)^{[\frac{t  }{C_{\Omega} \delta}] + 1} 
\lesssim \sqrt{\big[\frac{t  }{C_{\Omega} \delta}\big] + 1} (e \mathfrak{a})^{[\frac{t  }{C_{\Omega} \delta}] +1 }.
\Ee
Hence, we bound \eqref{sum_bound} by
\Be \notag
(\delta^{2\mathfrak{a}} e \mathfrak{a} )^{[\frac{t  }{C_{\Omega} \delta}] + 1} \sqrt{\big[\frac{t  }{C_{\Omega} \delta}\big] + 1} \lesssim e^{-t}.
\Ee
\end{proof}

\subsection{Estimates on Exponential Moments}
Now we are ready to prove Theorem \ref{theorem}.
First, we set 
\be
w (x, v) := e^{\theta (|v|^2+ 2\Phi (x))} \ \text{ and } \  w^\prime (x, v) := e^{\theta^\prime (|v|^2+ 2\Phi (x))},
\ee
where $0 \leq 2 \theta < \theta^\prime = \frac{1}{2}$.
{\color{black}
Suppose $(X,V)$ solves \eqref{characteristics}.
From \eqref{eq:energy conservation}, we have 
\[
\frac{\dd}{\dd s} \big( |V(s;t,x,v)|^2 / 2 + \Phi(X(s;t,x,v)) \big) = 0.
\]
This indicates that both $w (x, v)$ and $w^\prime (x, v)$ are constant along the the characteristic \eqref{characteristics}.
}

\begin{proof}[\textbf{Proof of Theorem \ref{theorem}}]

We start to prove \eqref{theorem_infty_1}, and pick $\varrho (t) = t + 1$ to utilize the $L^1$-decay of Theorem \ref{theorem_1}. Then we work on the stochastic cycle representation of $\varrho (t) w^\prime (x, v) f(t, x, v)$ in \eqref{expand_k1}-\eqref{expand_k4}.

For the contribution of \eqref{expand_k1}, since $t^1 < 0$, and $w^\prime, f$ are constant along the characteristic trajectory. Thus we deduce that
\be \label{est:expand_k1}
\begin{split}
w^\prime (x, v) f(t, x, v) 
= w^\prime (X(0; t, x, v), V(0; t, x, v)) f(0, X(0; t, x, v), V(0; t, x, v))
\leq \| w^\prime f(0) \|_{L^\infty_{x,v}}.
\end{split}
\ee

Now we bound the contribution of \eqref{expand_k2}. Since
$| n(x) \cdot v | \lesssim w^\prime (x, v) = \mu^{-1} (x, v)$, we derive
\Be \label{est:expand_k2}
\begin{split}
\frac{1}{\varrho (t)} | \eqref{expand_k2} |
& \lesssim \frac{k}{\varrho (t)}  \bigg( \sup_{i }    \int_{\prod_{j=1}^{k} \mathcal{V}_j}   
     \mathbf{1}_{t^{i+1} < 0 \leq t^{i }} 
      \dd \tilde{\Sigma}_{i}\bigg)
      \varrho(0)
   \| w^\prime f(0) \|_{L^\infty_{x,v}}   
\\&  \lesssim  \frac{k}{\varrho (t)}  \bigg(   \int_{n(x^{j}) \cdot v^{j} >0}
        \frac{ |n(x^{j}) \cdot v^{j}| }{w^\prime (x^j, v^{j})}\dd v^j \bigg)   \| w^\prime f(0) \|_{L^\infty_{x,v}}
\\& \lesssim  \frac{k}{\varrho (t)}  \| w^\prime f(0) \|_{L^\infty_{x,v}}.
\end{split}
\Ee

Applying Lemma \ref{lem:bound1} and Theorem \ref{theorem_1}, we bound the contribution of \eqref{expand_k3}. Since
$| n(x) \cdot v | \lesssim w^\prime (x, v) = \mu ^{-1} (x, v)$ and $\varrho^\prime = 1$, we have
\be \label{est:expand_k3}
\begin{split}
\frac{1}{\varrho (t)} |\eqref{expand_k3}|     
& \lesssim \frac{k}{\varrho (t)} \sup_{i }  \int_{\prod_{j=1}^{i} \mathcal{V}_j}  \mathbf{1}_{0 \leq t^{i}}
\int^{t^{i}}_{ \max(0, t^{i+1})} 
w^\prime (X(s; t^i, x^i, v^i), V(s; t^i, x^i, v^i))
\\& \hspace{7cm} \times f(s, X(s; t^i, x^i, v^i), V(s; t^i, x^i, v^i)) \dd s \dd \tilde{\Sigma}_{i}    
\\& \lesssim \frac{k}{\varrho (t)} \int^t_0 \| f(s) \|_{L^1_{x,v}} \dd s   
\lesssim \frac{k}{\varrho (t)} \times \| w^\prime f (0) \|_{L^\infty_{x,v}}.
\end{split}
\ee

Lastly we bound the contribution of \eqref{expand_k4}. From Lemma \ref{lem:small_largek}, we get
\Be \label{est:expand_k4}
\begin{split}
 \frac{1}{\varrho (t)} |\eqref{expand_k4}| 
& \lesssim \frac{\varrho (t^k)}{\varrho (t)} \sup_{(x,v) \in \bar{\O} \times \R^3}  \Big(\int_{\prod_{j=1}^{k -1} \mathcal{V}_j}   
    \mathbf{1}_{t_{k }(t,x,v,v^1,\cdots, v^{k-1}) \geq 0 }
\dd \sigma_1 \cdots \dd \sigma_{k-1}\Big)
 \| w^\prime f(t_k) \|_{L^\infty_{x,v}}
\\&  \lesssim e^{-t} \sup_{t \geq s \geq 0} \| w^\prime f(s) \|_{L^\infty_{x,v}}.
\end{split}
\Ee

Collecting estimates from \eqref{est:expand_k1}-\eqref{est:expand_k4} and using $k \lesssim t$, we derive 
\Be \label{k estimate}
(1 -  e^{-t}) \sup_{t \geq 0} \| w^\prime f(t) \|_{L^\infty_{x,v}}
\lesssim 
(1 + \frac{k}{\varrho (t)}) \times
\| w^\prime f (0) \|_{L^\infty_{x,v}}. 
\Ee
Therefore, we prove \eqref{theorem_infty_1}.

\medskip

Next, we prove \eqref{theorem_infty}. To show the decay of exponential moments and again utilize the $L^1$-decay, we set a new weight function
\Be \label{varrho}
\varrho(t):=  
(\ln\langle t\rangle )^{6 - \A} \langle t\rangle^{\A - 5}.
\Ee
Clearly we have 
$\varrho^\prime(t) \lesssim  (\ln\langle t\rangle )^{6 - \A} \langle t\rangle^{\A - 6}$ for $t\gg1$.

\medskip

\textbf{Step 1.} 
From Lemma \ref{sto_cycle}, we derive the form of $\int_{\R^3} w(x, v) |f (t, x, v)| \dd v$. First we split $|v| \geq t/2$ and $t_1 \leq 3t/4$ case to get \eqref{v<t/2 term in wf} and \eqref{first team in f_exp}. Next, for $t_1 \geq 3t/4$ case, we follow along the stochastic cycles twice with $k=2$, $t_* = t/2$, and obtain \eqref{second team in f_exp} and \eqref{f_exp2}.
\begin{align}
\int_{\R^3} w(x, v) |f(t,x,v)| \dd v 
& \leq 
\int_{|v| \geq t/2} w(x, v) |f(t,x,v)| \dd v
\label{v<t/2 term in wf}
\\& + \int_{|v| \leq t/2} \mathbf{1}_{t^1 \leq 3t/4}  w(x, v) |f(3t/4, X(3t/4; t, x, v), V(3t/4; t, x, v))| \dd v
\label{first team in f_exp}
\\& + \int_{\R^3}  \mathbf{1}_{t^1 \geq 3t/4} w (x, v) \mu (x^1, \vb)  \int_{\prod^2_{j=1} \mathcal{V}_j} \mathbf{1}_{t^2 < t/2 < t^1} w (x^1, v^1) |f(t^1, x^1, v^1)| \dd \Sigma_{1}^2 \dd v 
\label{second team in f_exp} 
\\& + \int_{\R^3} \mathbf{1}_{t^1 \geq 3t/4}  
   w(x, v) \mu (x^1, \vb) 
  \Big| \int_{\prod^2_{j=1} \mathcal{V}_j} \mathbf{1}_{t^2 \geq t/2} w (x^2, v^2)
f(t^2, x^2, v^2) \dd \Sigma_{2}^2 \Big| \dd v, \label{f_exp2}
\end{align} 
where $\dd  {\Sigma}^{2}_{1} = \dd \sigma_{2}  \frac{ \dd \sigma_{1}}{ \mu (x^2, v^1)w(x^1, v^1)} $ and $\dd  {\Sigma}^{2}_{2} = \frac{ \dd \sigma_{2}}{ \mu (x^3, v^2) w(x^2, v^2)} \dd \sigma_{1}$.

For \eqref{v<t/2 term in wf},
from the $L^\infty$-boundedness, $\Phi(x) |_{x \in \bar{\O}} \geq 0$, and $0 < w < w^\prime$, we derive that
\Be \label{v<t/2 part in wf}
\begin{split}
\int_{|v| \geq t/2} w(x, v) |f(t,x,v)| \dd v
& \leq \int_{|v| \geq t/2} \frac{w (x, v)}{w^\prime (x, v)} \dd v \| w^\prime f(0) \|_{L^\infty_{x,v}}
\\& \leq \int_{|v| \geq t/2} e^{- (\theta^\prime - \theta) |v|^2 } \dd v \| w^\prime f(0) \|_{L^\infty_{x,v}}
\\& \lesssim \frac{1}{\sqrt{\theta^\prime - \theta}} e^{- \frac{(\theta^\prime - \theta) t^2}{4}} \| w^\prime f(0) \|_{L^\infty_{x,v}}.
\end{split}
\Ee

For \eqref{first team in f_exp}, from $t^1 \leq 3t/4$ and Lemma \ref{conservative field}, we have for $t^1 \leq s \leq t$,
\Be \label{vvb relation}
\tb (x, v) = t - t^1 \geq t/4, \ \
\frac{|V(s;t,x,v)|^2}{2} + \Phi(X(s;t,x,v)) 
= \frac{|\vb|^2}{2}.
\Ee
On the other hand, using \eqref{tb estimate}, we get
\Be \label{tbvb relation}
\tb (x, v) \lesssim a^{\frac{1}{2} \vbn^2 (x, v)}.
\Ee
Then, from the $L^\infty$-boundedness, \eqref{vvb relation} and \eqref{tbvb relation}, we deduce that 
\Be \label{first part in wf}
\begin{split}
\eqref{first team in f_exp} 
& \lesssim \int_{|v| \leq t/2} \frac{w (x, v)}{w^\prime (X(3t/4; t, x, v), V(3t/4; t, x, v))} \dd v \| w^\prime f(0) \|_{L^\infty_{x,v}}
\\& \leq \int_{|v| \leq t/2} 
e^{(\theta - \theta^\prime) |\vb (x, v)|^2} \dd v \| w^\prime f(0) \|_{L^\infty_{x,v}}
\\& \leq \int_{|v| \leq t/2} 
e^{- \frac{1}{4} |\vb (x, v)|^2} \dd v \| w^\prime f(0) \|_{L^\infty_{x,v}}
\\& \leq \int_{|v_3| \leq t/2} 
|\tb (x, v) |^{- \frac{\A}{2}}  \dd v_3 \| w^\prime f(0) \|_{L^\infty_{x,v}}
\lesssim \langle t\rangle^{1 - \frac{\A}{2} } \| w^\prime f(0) \|_{L^\infty_{x,v}}.
\end{split}
\Ee

Next, we bound $\int_{\R^3} w(x, v) \mu (x^1, \vb) \dd v$ shown in \eqref{second team in f_exp} and \eqref{f_exp2}.
Note that from \eqref{vvb relation}, we have 
\be \notag
\mu^{-1} (x^1, \vb) = w^\prime (x, v).
\ee 
Thus, we derive
\Be \label{vvb estimate}
\int_{\R^3} w(x, v) \mu (x^1, \vb) \dd v 
= \int_{\R^3} \frac{w (x, v)}{w^\prime (x, v)} \dd v \lesssim_{\theta} 1.
\Ee

For \eqref{second team in f_exp}, since $\int_{\mathcal{V}_2} \dd \sigma_2$ is bounded and from \eqref{vvb estimate}, we have 
\Be \label{first step in second part in wf}
\eqref{second team in f_exp} 
\lesssim \int_{\mathcal{V}_1} \mathbf{1}_{t^2 < t/2 < t^1} |f(t^1, x^1, v^1)| \{ n(x^1) \cdot v^1 \} \dd v^1.
\Ee
From $t^1 \geq 3t/4$, $t^2 < t/2$ and \eqref{tb estimate}, we have 
\be
\tb (x^1, v^1) = t^1 - t^2 \geq t/4, \ \ 
a^{\frac{1}{2} (v^1_3)^2 } \gtrsim \tb (x^1, v^1).
\ee 
Then, from the $L^\infty$-boundedness and $0 < n(x^1) \cdot v^1 \lesssim e^{\e |v^1|^2} < w^\prime (x^1, v^1)$ for $0 < \e \ll 1/2$,
we derive
\Be \label{second part in wf}
\begin{split}
\eqref{first step in second part in wf} 
& \lesssim \int_{\mathcal{V}_1} \frac{ n(x^1) \cdot v^1 }{w^\prime (x^1, v^1)} \dd v^1 \| w^\prime f(0) \|_{L^\infty_{x,v}}
\\& \lesssim \int_{v^1_3 \leq 0} 
e^{(\e - \theta^\prime) |v^1_3|^2} \dd v^1_3 \| w^\prime f(0) \|_{L^\infty_{x,v}}
\\& \lesssim \int_{v^1_3 \leq 0} 
e^{- \frac{\theta^\prime}{2} |v^1_3|^2}
e^{(\e - \frac{\theta^\prime}{2}) |v^1_3|^2} \dd v^1_3 \| w^\prime f(0) \|_{L^\infty_{x,v}}
\\& \lesssim \int_{v^1_3 \leq 0} 
\big( \tb (x^1, v^1) \big)^{- \frac{\A}{2}}
e^{(\e - \frac{\theta^\prime}{2}) |v^1_3|^2} \dd v^1_3 \| w^\prime f(0) \|_{L^\infty_{x,v}}
\\& \lesssim \langle t\rangle^{- \frac{\A}{2}} \int_{v^1_3 \leq 0} e^{(\e - \frac{\theta^\prime}{2}) |v^1_3|^2} \dd v^1_3 \| w^\prime f(0) \|_{L^\infty_{x,v}}
\lesssim \langle t\rangle^{- \frac{\A}{2}} \| w^\prime f(0) \|_{L^\infty_{x,v}}.
\end{split}
\Ee

\medskip

\textbf{Step 2.}
Now we only need to bound \eqref{f_exp2}. Since $\int_{\R^3} w(x, v) \mu (x^1, \vb) \dd v \lesssim_{\theta} 1$, and $\int_{\mathcal{V}_1} \dd \sigma_1$ is bounded, it suffices to prove the decay of
\Be \label{decay of f_exp2}
\sup_{v \in \R^3, v^1 \in \mathcal{V}_1}
\Big| \int_{\mathcal{V}_2} 
\mathbf{1}_{t^2 \geq t/2} f(t^2,x^2,v^2) \{ n(x^{2}) \cdot v^{2} \} \dd v^{2} \Big|. 
\Ee 

Here we define $g (t, x, v) := \varrho (t) w (x, v) f (t, x, v)$, and note that
\Be \notag
\frac{1}{\varrho (t^2)} \int_{\mathcal{V}_2} \frac{|n(x^2) \cdot v^2|}{w (x^2, v^2)} g (t^2, x^2, v^2) \dd v^2 = \int_{\mathcal{V}_2} f(t^2,x^2,v^2) \{ n(x^2) \cdot v^2 \} \dd v^2.
\Ee 
Therefore, it suffices to show the decay of $\big| \frac{1}{\varrho (t^2)} \int_{\mathcal{V}_2} \mathbf{1}_{t^2 \geq t/2} \frac{|n(x^2) \cdot v^2|}{w (x^2, v^2)} g (t^2, x^2, v^2) \dd v^2 \big|$.

Applying Lemma \ref{sto_cycle_2} with $w(x, v) = e^{\theta (|v|^2+ 2\Phi (x))}$ and $\varrho(t)$ in \eqref{varrho}, and choosing $k \geq \mathfrak{C} t$ as in Lemma \ref{lem:small_largek}, we obtain the following stochastic cycle representation of 
$g (t^2, x^2, v^2) = \varrho (t^2) w (x^2, v^2) f (t^2, x^2, v^2)$:
\begin{align}
g (t^2, x^2, v^2) 
= \mathbf{1}_{t^3 < 0} \varrho & (0) w(x^2, v^2) f (0, X(0; t^2, x^2, v^2), V(0; t^2, x^2, v^2))
\label{expand_g1}
    \\& + w(x^2, v^2) \int^{t^2}_{\max(0, t^3)} \varrho^\prime(s) f(s, X(s; t^2, x^2, v^2), V(s; t^2, x^2, v^2)) \dd s \label{expand_g2}
    \\& +  w \mu (x^3, v^2_{\mathbf{b}}) \sum\limits^{k-1}_{i=3}  \int_{\prod_{j=3}^{i} \mathcal{V}_j} \Big\{ \mathbf{1}_{t^{i+1} < 0 \leq t^{i}}  \varrho (0) w(x^i, v^{i}) \notag
     \\& \hspace{5cm} \times f (0, X(0; t^i, x^i, v^i), V(0; t^i, x^i, v^i)) \Big\}
      \dd \tilde{\Sigma}_{i}
\label{expand_g3}
    \\& + w \mu (x^3, v^2_{\mathbf{b}}) \sum\limits^{k-1}_{i=3} \int_{\prod_{j=3}^{i} \mathcal{V}_j} \mathbf{1}_{0 \leq t^{i}}
     \Big\{ \int^{t^{i}}_{ \max(0, t^{i+1})} \varrho^\prime(s) w (x^i, v^{i}) 
     \notag
     \\& \hspace{5cm} \times f(s, X(s; t^i, x^i, v^i), V(s; t^i, x^i, v^i)) \dd s 
      \Big\} \dd \tilde{\Sigma}_{i}
\label{expand_g4}
    \\& + w \mu (x^3, v^2_{\mathbf{b}}) \int_{\prod_{j=3}^{k } \mathcal{V}_j}   
    \mathbf{1}_{t^{k} \geq 0} \
    g (t^{k}, x^{k}, v^{k})
     \dd \tilde{\Sigma}_{k}, \label{expand_g5} 
\end{align} 
where 
$\dd \tilde{\Sigma}_{i} 
:= \frac{ \dd \sigma_{i}}{\mu (x^{i+1}, v^{i}) w(x^i, v^{i})} \dd \sigma_{i-1} \cdots \dd \sigma_3$ with $3 \leq i \leq k$.
Here, we regard $t^2, x^2, v^2$ as free parameters and from Lemma \ref{conservative field}, we have 
$\mu (x^3, v^2_{\mathbf{b}}) = \mu (x^3, v^2)$.

\medskip

\textbf{Step 3.}
Next we estimate the contribution of \eqref{expand_g1}-\eqref{expand_g5} in 
$\frac{1}{\varrho (t^2)} \int_{\mathcal{V}_2} \frac{ |n(x^2) \cdot v^2| }{w (x^2, v^2)} g (t^2, x^2, v^2) \dd v^2$
term by term. 

We start with the contribution of \eqref{expand_g1}.
From $t^2 \geq t/2$ and $t^3 \leq 0$, we have 
\be \notag
\| w (x^2, v^2) f (t^2, x^2, v^2) \|_{L^\infty_{x,v}} \leq \| w (x, v) f(0, x, v) \|_{L^\infty_{x,v}}.
\ee 
From the $L^\infty$-boundedness and $0 < n(x^2) \cdot v^2 \lesssim w (x^2, v^2) < w^\prime (x^2, v^2)$,  
we deduce that 
\Be \label{est:expand_g1}
\begin{split}
\frac{1}{\varrho (t^2)} \int_{\mathcal{V}_2} \frac{|n(x^2) \cdot v^2|}{w (x^2, v^2)} |\eqref{expand_g1}| \dd v^2
& = \frac{1}{\varrho (t^2)} \int_{\mathcal{V}_2} |n(x^2) \cdot v^2| | \varrho (0) f (t^2, x^2, v^2) | \dd v^2 
\\& \lesssim \frac{1}{\varrho (t^2)} \int_{\mathcal{V}_2} \frac{|n(x^2) \cdot v^2|}{w (x^2, v^2)} \varrho(0) \dd v^2 \times \| w f(0) \|_{L^\infty_{x,v}}
\\& \lesssim \frac{1}{\varrho (t)} \varrho(0) \| w f(0) \|_{L^\infty_{x,v}}
\lesssim  \frac{1}{\varrho (t)} \| w^\prime f(0) \|_{L^\infty_{x,v}}.
\end{split}
\Ee

Now we bound the contribution of \eqref{expand_g2}. Recall Theorem \ref{theorem_1} with
$\varrho^\prime(t) \lesssim  (\ln\langle t\rangle )^{6 - \A} \langle t\rangle^{\A - 6}$ for $t\gg1$, and Lemma  \ref{lem:bound1}, we get 
\Be \label{est:expand_g2}
\begin{split}
\frac{1}{\varrho (t^2)} \int_{\mathcal{V}_2} \frac{|n(x^2) \cdot v^2|}{w (x^2, v^2)} |\eqref{expand_g2} | \dd v^2
& \lesssim \frac{1}{\varrho (t)} \int^t_0 \| \varrho^\prime (s) f(s) \|_{L^1_{x,v}} \dd s    
\\& \lesssim \frac{1}{\varrho (t)} \int^t_0 \| (\ln\langle s \rangle)^{6 - \A} \langle s \rangle^{\A - 6} f(s) \|_{L^1_{x,v}} \dd s
\\& \lesssim \frac{t}{\varrho (t)} \times \| w^\prime f (0) \|_{L^\infty_{x,v}}. 
\end{split}
\Ee

Next, we bound the contribution of \eqref{expand_g3}. 
From $t^{i+1} < 0 \leq t^{i }$, we have 
\be \notag
\| w (x^i, v^i) f (t^i, x^i, v^i) \|_{L^\infty_{x,v}} \leq \| w (x, v) f(0, x, v) \|_{L^\infty_{x,v}}.
\ee 
From the $L^\infty$-boundedness and 
$0 < n(x^2) \cdot v^2 \lesssim w (x^2, v^2) < w^\prime (x^2, v^2) < \mu^{-1} (x^3, v^2)$, we derive 
\Be \label{est:expand_g3}
\begin{split}
\frac{1}{\varrho (t^2)} \int_{\R^3} \frac{|n(x^2) \cdot v^2|}{w (x^2, v^2)} |\eqref{expand_g3}| \dd v^2
& \lesssim \frac{k}{\varrho (t)}  \bigg( \sup_{i }    \int_{\prod_{j=3}^{i} \mathcal{V}_j}   
     \mathbf{1}_{t^{i+1} < 0 \leq t^{i }} 
      \dd \tilde{\Sigma}_{i} \bigg)
      \varrho(0)
   \| w f(0) \|_{L^\infty_{x,v}}   
\\&  \lesssim  \frac{k}{\varrho (t)}  \bigg(   \int_{n(x^{i}) \cdot v^{i} >0}
        \frac{ |n(x^{i}) \cdot v^{i}| }{w (x^i, v^{i})}\dd v^i \bigg)   \| w f(0) \|_{L^\infty_{x,v}}
\\& \lesssim  \frac{k}{\varrho (t)}  \| w^\prime f(0) \|_{L^\infty_{x,v}}.
\end{split}
\Ee

Again using Lemma \ref{lem:bound1} and Theorem \ref{theorem_1}, we bound the contribution of \eqref{expand_g4}. From  $0 < n(x^2) \cdot v^2 \lesssim w (x^2, v^2) \leq \mu^{-1} (x^3, v^2)$ and $\varrho^\prime(t) \lesssim  (\ln\langle t\rangle )^{6 - \A} \langle t\rangle^{\A - 6}$ for $t \gg 1$, we have
\Be \label{est:expand_g4}
\begin{split}
& \ \ \ \ \frac{1}{\varrho (t^2)} \int_{\R^3} \frac{|n(x^2) \cdot v^2|}{w (x^2, v^2)} |\eqref{expand_g4}| \dd v^2   \\
& \lesssim \frac{k}{\varrho (t)} \times  \sup_{i }  \int_{\prod_{j=3}^{i} \mathcal{V}_j}  \mathbf{1}_{0 \leq t^{i}}
\int^{t^{i}}_{ \max(0, t^{i+1})}w (x^i, v^{i}) \varrho^\prime(s) f(s, X(s; t^i, x^i, v^i), V(s; t^i, x^i, v^i)) \dd s \dd \tilde{\Sigma}_{i}    
\\& \lesssim \frac{k}{\varrho (t)} \int^t_0 \| \varrho^\prime (s) f(s) \|_{L^1_{x,v}} \dd s     
\lesssim \frac{k}{\varrho (t)} \int^t_0 
\| (\ln\langle s \rangle)^{6 - \A} \langle s \rangle^{\A - 6} f(s) \|_{L^1_{x,v}} \dd s
\\& \lesssim \frac{k t}{\varrho (t)} \times \| w^\prime f (0) \|_{L^\infty_{x,v}}. 
\end{split}
\Ee

Lastly we bound the contribution of \eqref{expand_g5}.
Applying Lemma \ref{lem:small_largek} with $k \geq \mathfrak{C} t$, we get
\Be \label{est:expand_g5}
\begin{split}
& \ \ \ \ \frac{1}{\varrho (t^2)} \int_{\R^3} \frac{|n(x^2) \cdot v^2|}{w (x^2, v^2)} |\eqref{expand_g5}| \dd v^2\\
& \lesssim \frac{\varrho (t^k)}{\varrho (t^2)} \sup_{(x,v) \in \bar{\O} \times \R^3}  \Big(\int_{\prod_{j=3}^{k -1} \mathcal{V}_j}   
    \mathbf{1}_{t^{k }(t^2,x^2,v^2,\cdots, v^{k-1}) \geq 0 }
\dd \sigma_3 \cdots \dd \sigma_{k-1}\Big)
\sup_{t_k \geq 0} \| w f(t^k) \|_{L^\infty_{x,v}}
\\&  \lesssim e^{-t} \sup_{t_k \geq 0} \| w^\prime f(t^k) \|_{L^\infty_{x,v}}
 \lesssim e^{-t} \| w^\prime f(0) \|_{L^\infty_{x,v}}.
\end{split}
\Ee

Collecting estimates from \eqref{est:expand_g1}-\eqref{est:expand_g5} and using $k \lesssim t$, we derive 
\Be \label{g estimate}
\big| \frac{1}{\varrho (t^2)} \int_{\mathcal{V}_2} \mathbf{1}_{t^2 \geq t/2} \frac{|n(x^2) \cdot v^2|}{w (x^2, v^2)} g (t^2, x^2, v^2) \dd v^2 \big|
\leq 
\max \{ \frac{1}{\varrho (t)}, \frac{(k+1) t}{\varrho (t)}  , e^{-t} \}   
\lesssim \frac{\langle t\rangle^2}{\varrho (t)}.
\Ee
Using $\varrho(t) = (\ln\langle t\rangle )^{6 - \A} \langle t\rangle^{\A - 5}$, $0 < w (x, v) < \mu^{-1} (x, v)$ and \eqref{g estimate}, we conclude
\Be \notag
\eqref{f_exp2} 
\lesssim \frac{\langle t\rangle^2}{\varrho (t)}
\lesssim \langle t\rangle^{\A - 7}.
\Ee
The above estimate, together with \eqref{v<t/2 part in wf}, \eqref{first part in wf} and \eqref{second part in wf}, we prove \eqref{theorem_infty}.
\end{proof}

 \section*{Conflict of interest statement}On behalf of all authors, the corresponding author states that there is no conflict of interest.

 \section*{Acknowledgment}This project is partly supported by NSF-CAREER 2047681, Brain Pool fellowship, and Simons fellowship.

\end{document}